\documentclass[11pt,letterpaper]{article}
\usepackage{graphicx}
\usepackage{epsfig,subfigure}
\usepackage{amssymb}
\usepackage{amsthm}
\usepackage{amsfonts}
\usepackage{amsmath}
\usepackage{multicol,multirow}
\usepackage[margin=1in]{geometry}

\newtheorem{theorem}{Theorem}
\newtheorem{lemma}[theorem]{Lemma}
\newtheorem{corollary}[theorem]{Corollary}
\newtheorem{proposition}[theorem]{Proposition}
\newtheorem{definition}{Definition}

\newtheorem{example}[theorem]{Example}

\def\endproof{\vbox{\hrule height0.6pt\hbox{%
   \vrule height1.3ex width0.6pt\hskip0.8ex
   \vrule width0.6pt}\hrule height0.6pt
  }}

\newcommand{\ones}{\mathbf{1}}
\newcommand{\bQ}{\mathbf{Q}}

\newcommand{\bY}{\mathbf{Y}}

\newcommand{\bx}{\mathbf{x}}
\newcommand{\by}{\mathbf{y}}

\newcommand{\bz}{\mathbf{z}}
\newcommand{\bc}{\mathbf{c}}

\newcommand{\bu}{\mathbf{u}}
\newcommand{\bv}{\mathbf{v}}
\newcommand{\bw}{\mathbf{w}}

\newcommand{\sage}{\mathrm{SAGE}}
\newcommand{\age}{\mathrm{AGE}}
\newcommand{\diff}{\backslash}
\newcommand{\bzero}{\mathbf{0}}
\newcommand{\bones}{\mathbf{1}}
\newcommand{\R}{\mathbb{R}}
\newcommand{\Q}{\mathbb{Q}}

\newcommand{\K}{\mathcal{K}}
\newcommand{\bnu}{\boldsymbol \nu}
\newcommand{\blambda}{\boldsymbol \lambda}
\newcommand{\bdelta}{\boldsymbol \delta}
\newcommand{\bmu}{\boldsymbol \mu}
\newcommand{\bzeta}{\boldsymbol \zeta}

\newcommand{\balpha}{\boldsymbol \alpha}
\newcommand{\bgamma}{\boldsymbol \gamma}
\newcommand{\btau}{\boldsymbol \tau}

\title{Relative Entropy Relaxations for Signomial Optimization}

\author{Venkat Chandrasekaran$^c$ and Parikshit Shah$^w$ \thanks{Email: venkatc@caltech.edu, pshah@discovery.wisc.edu} \vspace{0.25in} \\ $^c$ Departments of Computing and Mathematical Sciences and of Electrical Engineering\\ California Institute of Technology \\ Pasadena, CA 91125  \vspace{0.1in} \\ $^w$ Wisconsin Institutes for Discovery\\ University of Wisconsin\\ Madison, WI 53715}
\date{September 25, 2014}

\begin{document}

\maketitle

\begin{abstract}
Signomial programs (SPs) are optimization problems specified in terms of signomials, which are weighted sums of exponentials composed with linear functionals of a decision variable.  SPs are non-convex optimization problems in general, and families of NP-hard problems can be reduced to SPs.  In this paper we describe a hierarchy of convex relaxations to obtain successively tighter lower bounds of the optimal value of SPs.  This sequence of lower bounds is computed by solving increasingly larger-sized relative entropy optimization problems, which are convex programs specified in terms of linear and relative entropy functions.  Our approach relies crucially on the observation that the relative entropy function -- by virtue of its joint convexity with respect to both arguments -- provides a convex parametrization of certain sets of globally nonnegative signomials with efficiently computable nonnegativity certificates via the arithmetic-geometric-mean inequality.  By appealing to representation theorems from real algebraic geometry, we show that our sequences of lower bounds converge to the global optima for broad classes of SPs.  Finally, we also demonstrate the effectiveness of our methods via numerical experiments.
\end{abstract}



\pagestyle{myheadings} \thispagestyle{plain} \markboth{V. Chandrasekaran and P. Shah} {Relative Entropy Relaxations for Signomial Optimization}

\section{Introduction} \label{sec:intro}

A \emph{signomial} is a weighted sum of exponentials composed with linear functionals of a variable $\bx \in \R^n$:
\begin{equation}
f(\bx) = \sum_{j=1}^\ell \bc_j \exp\left\{{\balpha^{(j)}}' \bx\right\} \label{eq:sumexp}
\end{equation}
Here $\bc_j \in \R$ and $\balpha^{(j)} \in \R^n$ are fixed parameters.\footnote{In the literature \cite{DufPZ1967}, signomials are typically parametrized somewhat differently as weighted sums of generalized ``monomials.''  A monomial consists of a product of variables, each raised to an arbitrary real power, and the variables only take on positive values.  It is straightforward to transform such functions to sums of exponentials of the type \eqref{eq:sumexp} discussed in this paper; see \cite{BoyKVH2007,DufPZ1967}.}  A signomial program (SP) is an optimization problem in which a signomial is minimized subject to constraints on signomials, all in a decision variable $\bx$ \cite{DufP1973,DufPZ1967}.  SPs are non-convex in general, and they include NP-hard problems as special cases \cite{Chi2005}.  Consequently, we do not expect to obtain globally optimal solutions of general SPs in a computationally tractable manner.  This paper describes a hierarchy of convex relaxations based on relative entropy optimization for obtaining lower bounds of the optimal value of an SP.

Geometric programs (GPs) constitute a prominent subclass of SPs in which a \emph{posynomial} -- a signomial with positive coefficients, i.e., the $\bc_j$'s are all positive -- is minimized subject to upper-bound constraints on posynomials \cite{BoyKVH2007,DufPZ1967}.  GPs are convex optimization problems that can be solved efficiently, and they have been successfully employed in a wide range of applications such as power control in communication systems \cite{Chi2005}, circuit design \cite{BoyKPH2005}, approximations to the matrix permanent \cite{LinSW2000}, and the computation of capacities of point-to-point communication channels \cite{ChiB2004}.  However, the additional flexibility provided by SPs via constraints on arbitrary signomials rather than just posynomials is useful in a range of problems to which GPs are not directly applicable.  Examples include resource allocation in networks \cite{Chi2005}, control problems involving chemical processes (see the extensive reference list in \cite{MarF1997}), spatial frame design \cite{YunX1997}, and certain nonlinear flow problems in graphs \cite{Pet1976}.

\subsection{Our Contributions}\label{subsec:contrib}
Central to the development in this paper is a view of global minimization that is grounded in duality: globally minimizing a signomial is computationally equivalent to the problem of certifying global nonnegativity of a signomial.  Although certifying global nonnegativity of a general signomial is a computationally intractable problem, one can appeal to GP duality \cite{DufPZ1967} to show that certifying nonnegativity of a signomial with all but one coefficient being positive can be accomplished in a computationally tractable manner.  These observations suggest a natural sufficient condition for certifying nonnegativity of general signomials via certificates that can be computed efficiently.  Specifically, if a signomial $f(\bx)$ can be decomposed as $f(\bx) = \sum_{i=1}^k f_i(\bx)$, where each $f_i(\bx)$ is a globally nonnegative signomial with at most one negative coefficient, then $f(\bx)$ is clearly nonnegative.  We refer to a decomposition of this form as a sum-of-AM/GM-exponentials (SAGE) decomposition, and to the associated functions $f(\bx)$ that are decomposable in this manner as SAGE functions.  The reason for this terminology is that certifying global nonnegativity of each of the functions $f_i(\bx)$ is accomplished by verifying an appropriate arithmetic-geometric-mean (AM/GM) inequality as described in Section~\ref{subsec:age}.

The key insight underlying our approach in Section~\ref{sec:sage} is that computing a SAGE decomposition of $f(\bx)$ can be cast as a feasibility problem involving a system of linear inequalities as well as inequalities specifying upper bounds on \emph{relative entropy} functions.  Recall that the relative entropy function is defined as follows for $\bnu, \blambda$ in the nonnegative orthant $\R^\ell_+$:
\begin{equation}
D(\bnu, \blambda) = \sum_{j=1}^\ell \bnu_j \log\left(\frac{\bnu_j}{\blambda_j}\right). \label{eq:relentdef}
\end{equation}
This function is jointly convex in both arguments \cite[p.90]{BoyV2004}, and the associated feasibility problem of finding a SAGE decomposition can be solved efficiently via convex optimization.  As discussed in Section~\ref{subsec:unconstrained}, this method can be employed to obtain lower bounds on the original signomial $f(\bx)$ by solving a tractable convex program involving linear as well as relative entropy functions.

In Section~\ref{sec:hierarchy} we also describe a principled framework to obtain a family of increasingly tighter lower bounds for general (constrained and unconstrained) SPs.  This family of bounds is obtained by solving hierarchies of successively larger convex programs based on relative entropy optimization; these hierarchies are derived by considering a sequence of tighter nonnegativity certificates for a signomial over a set defined by signomial constraints.  A prominent example of hierarchies of tractable convex programs being employed for intractable problems is in the setting of polynomial optimization problems, for which SDP relaxations have been developed based on sum-of-squares techniques \cite{Las2001,Par2000,Par2003}.  However, those methods are not directly relevant to SPs for several reasons, and we highlight these distinctions in Section~\ref{subsec:polyopt}.

The hierarchy of convex relaxations that we describe in Section~\ref{sec:hierarchy} has several notable features.  First, GPs are solved exactly by the first level in this hierarchy; thus, our hierarchy has the desirable property that ``easy problem instances remain easy.''  Second, the family of lower bounds is invariant under a natural transformation of the problem data.  Specifically, the optimal value of an SP remains unchanged under the application of a non-singular linear transformation simultaneously to all the parameters $\balpha^{(j)}$ that appear in the exponents of the signomials in an SP (both in the objective and in the constraints).  The hierarchy of relative entropy relaxations described in Section~\ref{sec:hierarchy} leads to bounds that are invariant under such transformations.  Third, it is desirable that any procedure for obtaining lower bounds of the optimal value of an SP be robust to small perturbations of the exponents $\balpha^{(j)}$ in an SP.  As discussed in Section~\ref{subsec:polyopt}, this is also a feature of our proposed approach.  Fourth, for broad families of SPs our approach leads to a convergent sequence of lower bounds, i.e., our hierarchy approximates the optimal value of the SP arbitrarily well from below (see Theorems~\ref{theorem:unc} and \ref{theorem:con}).  Such a property of a hierarchy is usually referred to as \emph{completeness}.  At various stages in the paper we discuss numerical experiments that illustrate the effectiveness of our methods.

\paragraph{Related work} Several researchers have developed strategies for optimizing SPs based on variants of branch-and-bound methods \cite{MarF1997} as well as heuristic techniques based on successive approximations that can be solved via linear or geometric programming \cite{BoyKVH2007,Chi2005,She2005,She1998}.  The framework presented in this paper is qualitatively different as it is based on solving convex optimization problems involving linear and relative entropy functions to obtain guaranteed bounds on the optimal value of SPs.  Our approach is founded on the insight that the joint convexity of the relative entropy function leads to an effective convex parametrization of certain families of globally nonnegative signomials.

\subsection{Paper Outline} \label{subsec:outline}
Section~\ref{sec:sage} describes the SAGE decomposition for certifying nonnegativity of a signomial based on relative entropy optimization and its connection to the AM/GM inequality.  This approach underlies the subsequent development in Section~\ref{sec:hierarchy} in which a hierarchy of relative entropy relaxations is proposed for general SPs.   In Section~\ref{sec:completeness}, we provide theoretical support for our hierarchies of relative entropy relaxations by proving that they provide a sequence of lower bounds that converges to the global minimum for a broad class of SPs.  Throughout the paper, we describe numerical experiments in which relative entropy optimization techniques are employed to obtain bounds on some stylized SPs.  These results were obtained using a basic interior-point solver written in MATLAB by following the discussions in \cite{BoyV2004}.  We conclude with a discussion of potential future directions in Section~\ref{sec:discussion}.  We prove Theorems~\ref{theorem:unc} and \ref{theorem:con} in the Appendix.

\paragraph{Notation} In $\ell$-dimensional space we denote the nonnegative orthant by $\R^\ell_+$, the positive orthant by $\R^\ell_{++}$, rational vectors by $\Q^\ell$, and nonnegative rational vectors by $\Q^\ell_+$.  We denote the nonnegative integers by $\mathbb{Z}_+$ and the positive integers by $\mathbb{Z}_{++}$.  Given a vector $\bw \in \R^\ell$, we denote the vector formed by removing the $i$'th coordinate from $\bw$ by $\bw_{\diff i} \in \R^{\ell-1}$.  By convention, we assume that $0 \log{\frac{0}{\lambda}} = 0$ for any $\lambda \in \R_+$, and that $\nu \log{\frac{\nu}{0}}$ equals $0$ if $\nu = 0$ and $\infty$ if $\nu > 0$.  Given a collection of vectors $\{\balpha^{(j)}\}_{j=1}^\ell \subset \R^n$, we define the set $E_p(\{\balpha^{(j)}\}_{j=1}^\ell) \subset \R^n$ for $p \in \mathbb{Z}_+$ as follows:
\begin{equation*}
E_p(\{\balpha^{(j)}\}_{j=1}^\ell) = \left\{ \sum_{j=1}^\ell \lambda_j \balpha^{(j)} ~|~ \sum_{j=1}^\ell \lambda_j \leq p, ~ \lambda_j \in \mathbb{Z}_+ \right\}.
\end{equation*}

\section{SAGE Decomposition and Relative Entropy Optimization} \label{sec:sage}

In this section we describe a sufficient condition to certify global nonnegativity of a signomial $f$ by decomposing it into a sum of terms $f = \sum_{i} f_i$, where each $f_i$ has an efficiently computed nonnegativity certificate based on relative entropy optimization.  This method serves as the basic building block in the next section on hierarchies of relative entropy relaxations for general SPs.

\subsection{Signomials with One Negative Term and the AM/GM Inequality} \label{subsec:age}

The main observation underlying our methods is that certifying the nonnegativity of a signomial with at most one negative coefficient can be accomplished efficiently.

\begin{definition} \label{defn:age}
A globally nonnegative signomial with at most one negative coefficient is called an \emph{AM/GM-exponential}.
\end{definition}

The reason for this terminology is that certifying the nonnegativity of an AM/GM-exponential $g(\bx) = \beta \exp\{\balpha' \bx\} + \sum_{j=1}^\ell \bc_j \exp\{[\bQ' \bx]_j \}$ -- here, $\bc \in \R^\ell_+$, $\beta \in \R$, $\bQ \in \R^{n \times \ell}$, and $\balpha \in \R^n$ -- is essentially \emph{equivalent} to verifying an AM/GM inequality.  To see this connection, suppose there exists a $\bnu \in \R^\ell$ satisfying the following conditions:
\begin{equation}
D(\bnu,e\bc) - \beta \leq 0, ~ \bnu \in \R^\ell_+, ~ \bQ \bnu = (\ones' \bnu) \balpha. \label{eq:relent}
\end{equation}
A $\bnu \in \R^\ell_+ \backslash \bzero$ satisfying these conditions explicitly certifies the nonnegativity of $g(\bx)$ via the following two sequences of inequalities.\footnote{The conditions \eqref{eq:relent} being satisfied with $\bnu = \bzero$ corresponds to the situation that $\beta \geq 0$.}  First, we have that:
\begin{eqnarray*}
\sum_{j=1}^\ell \bc_j \exp\{[\bQ' \bx]_j \} &\stackrel{(i)}\geq& \prod_{j=1}^\ell \left( \frac{\bc_j \exp\{[\bQ' \bx]_j \}}{\bnu_j / (\ones' \bnu)} \right)^{\tfrac{\bnu_j}{\ones' \bnu}} \stackrel{(ii)}= \prod_{j=1}^\ell \left( \frac{\bc_j}{\bnu_j / (\ones' \bnu)} \right)^{\tfrac{\bnu_j}{\ones' \bnu}} \exp\{\balpha' \bx\}.
\end{eqnarray*}
Here inequality $(i)$ is a consequence of the AM/GM inequality applied with weights $\tfrac{\bnu}{\ones' \bnu}$, and equation $(ii)$ follows from \eqref{eq:relent}.  Next, we observe that:
\begin{eqnarray*}
\prod_{j=1}^\ell \left( \frac{\bc_j}{\bnu_j / (\ones' \bnu)} \right)^{\tfrac{\bnu_j}{\ones' \bnu}} &\stackrel{(iii)}=& \exp\left\{ -D(\tfrac{\bnu}{\ones' \bnu}, \bc) \right\} \stackrel{(iv)}\geq -\xi [D(\tfrac{\bnu}{\ones' \bnu},\bc) + \log(\xi) - 1 ], ~ \forall \xi \in \R_+ \\ &\stackrel{(v)}=& -\xi D(\tfrac{\bnu}{\ones' \bnu}, \tfrac{e}{\xi}\bc), ~ \forall \xi \in \R_+ \stackrel{(vi)}= -D(\tfrac{\xi \bnu}{\ones' \bnu}, e\bc), ~ \forall \xi \in \R_+ \\ &\stackrel{(vii)}\geq& -D(\bnu, e\bc) \stackrel{(viii)}\geq -\beta \exp\{\balpha' \bx \}.
\end{eqnarray*}
Here equation $(iii)$ follows from the definition of the relative entropy function \eqref{eq:relentdef}, inequality $(iv)$ follows from the observation that the exponential function is the convex conjugate of the negative entropy function --- i.e., $\exp\{ -\rho\} = \sup_{\xi \in \R_+} -\xi [\rho + \log(\xi) - 1 ]$ --- equations $(v), (vi)$ follow from the properties of the relative entropy function \eqref{eq:relentdef} (in particular, by noting that this function is positively homogenous), inequality $(vii)$ follows by setting $\xi = \ones' \bnu$, and finally inequality $(viii)$ follows from \eqref{eq:relent}.  Hence the existence of $\bnu \in \R^\ell$ satisfying \eqref{eq:relent} certifies the nonnegativity of the signomial $g(\bx)$.  The preceding discussion parallels some of the early expositions on GP duality based on the AM/GM inequality \cite{DufPZ1967}, although our formulation \eqref{eq:relent} is parametrized differently compared to the following formulation in the GP literature \cite{Chi2005,DufPZ1967}:
\begin{equation}
D(\bnu,\bc) + \log(- \beta) \leq 0, ~ \bnu \in \R^\ell_+, \ones' \bnu = 1, ~ \bQ \bnu = \balpha. \label{eq:relentgp}
\end{equation}
In comparison to \eqref{eq:relent}, the parametrization \eqref{eq:relentgp} is not jointly convex in $(\bnu,\bc,\beta)$.  As discussed after Lemma~\ref{lemma:agerelent}, the joint convexity of the inequality $D(\bnu,e\bc) - \beta \leq 0$ in \eqref{eq:relent} with respect to $(\bnu,\bc,\beta)$ plays a crucial role in our subsequent development.

\begin{example} \label{example:basicage} Consider the function $g(\bx) = \exp\{\bx_1\} + \exp\{\bx_2\} - \exp\{(\tfrac{1}{2}+\epsilon) \bx_1 +(\tfrac{1}{2}-\epsilon) \bx_2 \}$ for a fixed $\epsilon \in [-\tfrac{1}{2},\tfrac{1}{2}]$.  It is easily seen that $\bnu = (\tfrac{1}{2}+\epsilon, \tfrac{1}{2}-\epsilon)'$ satisfies the conditions \eqref{eq:relent}:
\begin{eqnarray*}
D((\tfrac{1}{2}+\epsilon, \tfrac{1}{2}-\epsilon)', e \bones) - (-1) &=& D((\tfrac{1}{2}+\epsilon, \tfrac{1}{2}-\epsilon)', \bones) \leq 0 \\ I_{2 \times 2} ~ \bnu &=& (\tfrac{1}{2}+\epsilon, \tfrac{1}{2}-\epsilon)' 
\end{eqnarray*}
Here $I_{2 \times 2}$ is the $2 \times 2$ identity matrix and the first inequality follows from the fact that $D((\tfrac{1}{2}+\epsilon, \tfrac{1}{2}-\epsilon)', \bones)$ is the negative entropy of a probability vector, which is always nonpositive \cite[p.15]{CovT2006}.  Consequently, $\bnu$ certifies the nonnegativity of $g(\bx)$.
\end{example}

Next, we show the converse result: if $g$ is an AM/GM-exponential, then there exists a certificate $\bnu$ satisfying \eqref{eq:relent}.  This follows as a consequence of strong duality applied to a suitable convex program.

\begin{lemma} \label{lemma:agerelent}
Let $g(\bx) = \beta \exp\{\balpha' \bx\} + \sum_{j=1}^\ell \bc_j \exp\{[\bQ' \bx]_j \}$, where $\bc \in \R^\ell_+$, $\beta \in \R$, $\bQ \in \R^{n \times \ell}$, and $\balpha \in \R^n$.  If $g(\bx) \geq 0$ for all $\bx \in \R^n$ then there exists $\bnu  \in \R^\ell$ satisfying the conditions \eqref{eq:relent}.
\end{lemma}

\begin{proof}
The signomial $g(\bx)$ being globally nonnegative is equivalent to the signomial $g(\bx) \exp\{-\balpha' \bx\}$ being globally nonnegative, which in turn is equivalent to $\sum_{j=1}^\ell \bc_j \exp\{[\bQ' \bx]_j - \balpha' \bx\} \geq -\beta$ for all $\bx \in \R^n$.  To certify a lower bound on the convex function $\sum_{j=1}^\ell \bc_j \exp\{[\bQ' \bx]_j - \balpha' \bx\}$, consider the following convex program:
\begin{eqnarray*}
\inf_{\bx \in \R^n, ~ \mathbf{t} \in \R^\ell} \bc' \mathbf{t} ~~~ \mathrm{s.t.} ~~~ \exp\left\{[\bQ' \bx]_j - \balpha' \bx\right\} \leq \mathbf{t}_j, ~ j = 1,\dots,\ell.
\end{eqnarray*}
This primal problem satisfies Slater's conditions \cite{Roc1970}, and its Lagrangian dual is:
\begin{eqnarray*}
\sup_{\bnu \in \R^\ell} -D(\bnu,e\bc) ~~~ \mathrm{s.t.} ~~~ \bnu \in \R^\ell_+, ~ \bQ \bnu = (\ones' \bnu) \balpha.
\end{eqnarray*}
The result follows as strong duality holds and the dual optimum is attained.
\end{proof}

In describing these results, we have implicitly viewed $\bc \in \R^\ell_+, \beta \in \R$ as \emph{fixed} coefficients.  However, the conditions \eqref{eq:relent} are \emph{convex} with respect to $\bc, \beta$ by virtue of the joint convexity of the relative entropy function.  Consequently, the \emph{set} of AM/GM-exponentials with respect to a collection of exponents can be efficiently parametrized as a convex set given by linear and relative entropy inequalities:
\begin{eqnarray}
&& \mathcal{C}_\age\big(\balpha^{(1)},\dots,\balpha^{(\ell)}; \balpha^{(0)} \big) \nonumber \\&& = \Big\{(\bc,\beta) \in \R^{\ell}_+ \times \R ~|~ g(\bx) = \beta \exp\{{\balpha^{(0)}}' \bx\} + \sum_{j=1}^\ell \bc_j \exp\{{\balpha^{(j)}}' \bx \}, \nonumber \\ && \hspace{1.33in} g(\bx) ~\text{an AM/GM-exponential} \Big\} \nonumber \nonumber \\ && = \Big\{(\bc,\beta) \in \R^{\ell}_+ \times \R ~|~ \exists \bnu \in \R^\ell_+ ~\mathrm{s.t.}~  D(\bnu,e\bc) \leq \beta, \sum_{j=1}^\ell \balpha^{(j)} \bnu_j = (\ones' \bnu) \balpha^{(0)} \Big\}. \label{eq:cage}
\end{eqnarray}
Based on this representation, one can see that the set $\mathcal{C}_\age\big(\balpha^{(1)},\dots,\balpha^{(\ell)}; \balpha^{(0)} \big)$ is a convex cone.  This parametrization plays a crucial role in the next subsection and beyond as we describe computationally efficient methods based on relative entropy optimization (\emph{jointly} over both arguments) to obtain bounds for SPs.

\subsection{SAGE Decomposition of a General Signomial} \label{subsec:sage}
If a signomial consists of more than one negative term, a natural sufficient condition for certifying nonnegativity is to express the signomial as a sum of AM/GM-exponentials:
\begin{definition} \label{defn:sage}
Given a finite collection of vectors $M \subset \R^n$, the set of \emph{sum-of-AM/GM-exponentials} with respect to $M$ is defined as:
\begin{equation*}
\sage(M) = \Big\{f ~|~ f = \sum_{i=1}^m f_i, ~ f_i ~\text{an AM/GM-exponential with exponents in} ~ M \Big\}.
\end{equation*}
A signomial $f \in \sage(M)$ is called a \emph{$\sage$ function} with respect to $M$, and a decomposition of $f$ as a sum of AM/GM-exponentials is called a \emph{SAGE decomposition}.
\end{definition}

As each term in a SAGE decomposition is globally nonnegative, it is easily seen that SAGE functions are globally nonnegative signomials.

\begin{figure}
\centering
\includegraphics[scale=0.2]{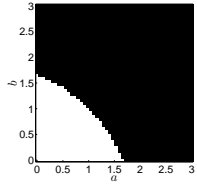} \caption{The locations in white denote those values of $(a,b) \in \R^2_+$ for which the signomial $f_{a,b}(\bx_1, \bx_2) = \exp\{\bx_1\} + \exp\{\bx_2\} - a \exp\{\delta \bx_1 + (1-\delta) \bx_2\} - b \exp\{(1-\delta) \bx_1 + \delta\bx_2\}$ with $\delta = \pi / 4$ is a SAGE function.  The intercept is around $1.682$.}
\label{fig:sageplot}
\end{figure}

\begin{example} \label{example:basicsage}
Consider a signomial $f(\bx) = \sum_i p_i(\bx) g_i(\bx)$, where each $p_i(\bx)$ is a posynomial and each $g_i(\bx)$ is an AM/GM-exponential.  If each $p_i(\bx),g_i(\bx)$ is defined with respect to exponents $\{\balpha^{(j)}\}_{j=1}^\ell$, one can check that $f$ is a SAGE function with respect to exponents $E_2(\{\balpha^{(j)}\}_{j=1}^\ell)$.
\end{example}

\begin{example} \label{example:basicsageparam}
As another example, consider the parametrized family of signomials $f_{a,b}(\bx_1, \bx_2) = \exp\{\bx_1\} + \exp\{\bx_2\} - a \exp\{\delta \bx_1 + (1-\delta) \bx_2\} - b \exp\{(1-\delta) \bx_1 + \delta\bx_2\}$, where we set $\delta = \pi / 4$.  Figure~\ref{fig:sageplot} shows a plot of the values of $(a,b) \in \R^2_+$ for which $f_{a,b}(\bx_1,\bx_2)$ is a SAGE function with respect to the exponents $\{(1,0)', (0,1)', (\delta,1-\delta)',(1-\delta,\delta)'\}$.
\end{example}

Although SAGE functions are globally nonnegative, not all nonnegative signomials are decomposable as SAGE functions.  The following simple example illustrates the point that SAGE decomposability is only a sufficient condition for nonnegativity:

\begin{example}
Consider $f(\bx_1,\bx_2,\bx_3) = \exp\{2\bx_1\} + \exp\{2\bx_2\} + \exp\{2\bx_3\} - 2 \exp\{\bx_1 + \bx_2\} - 2\exp\{\bx_1 + \bx_3\} + 2\exp\{\bx_2+\bx_3\}$.  This function is not a SAGE function with respect to the six exponents in the expression.  One can check that $f(\bx_1,\bx_2,\bx_3) = (\exp\{\bx_1\} - \exp\{\bx_2\} - \exp\{\bx_3\})^2$, which demonstrates its nonnegativity.  In Section~\ref{subsec:polyopt}, we contrast sum-of-squares methods from the polynomial optimization literature \cite{Las2001,Par2000,Par2003} with the methods developed in this paper, highlighting the benefits of convex relaxations based on SAGE decompositions for SPs.
\end{example}

For $M = \{\balpha^{(j)}\}_{j=1}^\ell$, we have that:
\begin{equation}
\begin{aligned}
\sage\left(\{\balpha^{(j)}\}_{j=1}^\ell \right) = \bigg\{f ~|~ & f(\bx) = \sum_{i=1}^\ell f_i(\bx), ~\mathrm{each}~ f_i(\bx) \geq 0 ~ \mathrm{for ~ all} ~\bx \in \R^n \\ & f_i(\bx) = \sum_{j=1}^\ell \tilde{\bc}_j^{(i)} \exp\big\{{\balpha^{(j)}}' \bx\big\} ~ \mathrm{with} ~ \tilde{\bc}_{\diff i}^{(i)} \in \R^{\ell-1}_+ \bigg\}. \label{eq:sagefinite}
\end{aligned}
\end{equation}
In order to obtain an efficient description of $\sage\left(\{\balpha^{(j)}\}_{j=1}^\ell \right)$ we consider the following set of coefficients (analogous to $\mathcal{C}_\age\big(\balpha^{(1)},\dots,\balpha^{(\ell)}; \balpha^{(0)} \big)$ in \eqref{eq:cage}):
\begin{equation}
\mathcal{C}_\sage\left(\balpha^{(1)}, \dots, \balpha^{(\ell)} \right) = \Big\{\bc \in \R^\ell ~|~ \sum_{j=1}^\ell \bc_j \exp\{{\balpha^{(j)}}' \bx \} \in \sage\Big(\{\balpha^{(j)}\}_{j=1}^\ell\Big) \Big\}. \label{eq:sagecoeff1}
\end{equation}
The next proposition summarizes the relevant properties of SAGE functions by giving an explicit representation of the set $\mathcal{C}_\sage\left(\balpha^{(1)}, \dots, \balpha^{(\ell)} \right)$.

\begin{proposition} \label{prop:sagecone}
Fix a collection of exponents $\{\balpha^{(j)}\}_{j=1}^\ell \subset \R^n$.  Then we have:
\begin{eqnarray}
&\mathcal{C}_\sage\big(&\balpha^{(1)}, \dots, \balpha^{(\ell)} \big) \nonumber \\ && = \Bigg\{\bc \in \R^\ell ~|~ \exists \bc^{(j)} \in \R^\ell, ~ j=1,\dots,\ell ~ \mathrm{s.t.} ~ \bc = \sum_{j=1}^{\ell} \bc^{(j)}, \nonumber \\ && \hspace{0.72in} {\bc^{(j)}_{\diff j} \choose \bc^{(j)}_j} \in \mathcal{C}_\age\Big(\balpha^{(1)},\dots,\balpha^{(j-1)},\balpha^{(j+1)},\dots,\balpha^{(\ell)}; \balpha^{(j)}\Big) \Bigg\} \nonumber \\ && = \Big\{\bc \in \R^\ell ~|~ \exists \bc^{(j)} \in \R^\ell, \bnu^{(j)} \in \R^\ell, ~ j=1,\dots,\ell ~ \mathrm{s.t.} ~\bc = \sum_{j=1}^\ell \bc^{(j)} \nonumber \\ && \hspace{0.72in} \sum_{i=1}^\ell \balpha^{(i)} \bnu^{(j)}_i = \bzero, ~ \bnu^{(j)}_j = -\ones' \bnu^{(j)}_{\diff j}, ~ j=1,\dots,\ell \nonumber \\ && \hspace{0.72in} D\left(\bnu^{(j)}_{\diff j}, e \bc^{(j)}_{\diff j}\right) - \bc^{(j)}_j \leq 0, ~ \bnu^{(j)}_{\diff j} \in \R^{\ell-1}, ~ j=1,\dots,\ell \Big\}.
\label{eq:sagecoeff2}
\end{eqnarray}
The set $\mathcal{C}_\sage\left(\balpha^{(1)}, \dots, \balpha^{(\ell)} \right)$ is a convex cone, and its dual is given by:
\begin{equation}
\begin{aligned}
\mathcal{C}^\star_\sage\Big(\balpha^{(1)}, \dots, \balpha^{(\ell)} \Big) = \Big\{\bv \in \R^\ell_+ ~|~ & \exists \btau^{(j)} \in \R^n, ~ j=1,\dots,\ell ~ \mathrm{s.t.} \\ & \bv_i \log\Big(\frac{\bv_i}{\bv_j}\Big) \leq \Big(\balpha^{(i)} - \balpha^{(j)}\Big)' \btau^{(i)}, ~ \forall i,j \Big\}.
\end{aligned}\label{eq:sagedualcone}
\end{equation}
\end{proposition}

\begin{proof}
For a coefficient vector $\bc \in \R^\ell$ to belong to $\mathcal{C}_\sage\left(\balpha^{(1)}, \dots, \balpha^{(\ell)} \right)$, we have from \eqref{eq:sagecoeff1} that $f(\bx) = \sum_{j=1}^\ell \bc_j \exp\left\{{\balpha^{(j)}}' \bx \right\}$ must be a SAGE function with respect to the exponents $\{\balpha^{(j)}\}_{j=1}^\ell$.  The description of $\mathcal{C}_\sage\left(\balpha^{(1)}, \dots, \balpha^{(\ell)} \right)$ in \eqref{eq:sagecoeff2} then follows directly from Lemma~\ref{lemma:agerelent} and from \eqref{eq:sagefinite}.

It is also clear that $\mathcal{C}_\sage\left(\balpha^{(1)}, \dots, \balpha^{(\ell)} \right)$ is a convex cone, because it can be viewed as the projection of a convex cone from \eqref{eq:sagecoeff2}.  Specifically, the inequality $D\left(\bnu^{(j)}_{\diff j}, e \bc^{(j)}_{\diff j}\right) - \bc^{(j)}_j \leq 0$ defines a convex cone because the relative entropy function is positively homogenous and convex jointly in both arguments, and all the other constraints define convex cones in the variables $\bc, \bc^{(j)}, \bnu^{(j)}$.

Finally, the description of the dual cone $\mathcal{C}^\star_\sage\left(\balpha^{(1)}, \dots, \balpha^{(\ell)} \right)$ follows from a straightforward calculation based on the observation that $\mathcal{C}_\sage\left(\balpha^{(1)}, \dots, \balpha^{(\ell)} \right)$ is a sum of convex cones $\mathcal{C}_\age\Big(\balpha^{(1)},\dots,\balpha^{(j-1)},\balpha^{(j+1)},\dots,\balpha^{(\ell)}; \balpha^{(j)}\Big)$ corresponding to AM/GM exponentials; therefore, the dual $\mathcal{C}^\star_\sage\left(\balpha^{(1)}, \dots, \balpha^{(\ell)} \right)$ is the intersection of the duals $\mathcal{C}^\star_\age\Big(\balpha^{(1)},\dots,\balpha^{(j-1)},\balpha^{(j+1)},\dots,\allowbreak\balpha^{(\ell)}; \balpha^{(j)}\Big)$.
\end{proof}

Thus, the set of SAGE functions with respect to a set of exponents can be efficiently characterized based on this proposition.  In Section~\ref{sec:hierarchy}, we employ this result to develop computationally tractable methods for obtaining lower bounds on signomials.  The characterization of the cone $\mathcal{C}_\sage\left(\balpha^{(1)}, \dots, \balpha^{(\ell)} \right)$ in Proposition~\ref{prop:sagecone} also clarifies an appealing invariance property of SAGE functions:

\begin{corollary} \label{corl:invariance}
Fix a collection of exponents $\{\balpha^{(j)}\}_{j=1}^\ell \subset \R^n$.  For any non-singular matrix $M \in \R^{n \times n}$, we have that
\begin{equation*}
\mathcal{C}_\sage\left(\balpha^{(1)}, \dots, \balpha^{(\ell)} \right) = \mathcal{C}_\sage\left(M \balpha^{(1)}, \dots, M\balpha^{(\ell)} \right).
\end{equation*}
\end{corollary}
\begin{proof}
This result follows by noting that the only appearance of the exponents $\{\balpha^{(j)}\}_{j=1}^\ell$ in the description of $\mathcal{C}_\sage\left(\balpha^{(1)}, \dots, \balpha^{(\ell)} \right)$ in \eqref{eq:sagecoeff2} is in the equation $\sum_{i=1}^\ell \balpha^{(i)} \bnu^{(j)}_i = \bzero$.  A $\bnu^{(j)} \in \R^\ell$ satisfies this constraint if and only if the same $\bnu^{(j)}$ also satisfies $\sum_{i=1}^\ell M \balpha^{(i)} \bnu^{(j)}_i = \bzero$.  This concludes the proof.
\end{proof}

Notice that the nonnegativity of a signomial is preserved under the application of a non-singular linear transformation simultaneously to all the exponents; more generally, the optimal value of an SP is unchanged under a non-singular transformation applied simultaneously to all the exponents in the objective and all the constraints.  The above invariance property of $\mathcal{C}_\sage\left(\balpha^{(1)}, \dots, \balpha^{(\ell)} \right)$ ensures that our lower bounds for SPs in Section~\ref{sec:hierarchy} based on SAGE decompositions are also similarly invariant under a simultaneous linear transformation applied to all the exponents in the SP.

\subsection{Remarks on Nonnegative Signomials} \label{subsec:nngsig}
Consider a signomial $f(\bx) \allowbreak = \allowbreak \sum_{j=1}^\ell \bc_j \exp\{{\balpha^{(j)}}' \bx\}$ with respect to exponents $\{\balpha^{(j)}\}_{j=1}^\ell$, and let the coefficients $\bc_j$ be nonzero for each $j=1,\dots\ell$.  If $f(\bx) \geq 0$ for all $\bx \in \R^n$, then we can make some observations about the coefficients $\bc_j$'s based on the structure of the set of exponents.  Specifically, suppose $\balpha^{(1)}$ is an extreme point of the convex hull of the exponents $\mathrm{conv}(\{\balpha^{(j)}\}_{j=1}^\ell)$; we will denote such an exponent as an \emph{extremal exponent}.  Then we can conclude from the separation theorem in convex analysis \cite{Roc1970} that there exists $\bu \in \R^n$ such that $\bu' \balpha^{(1)} = 1$ and $\bu' \balpha^{(j)} < 1$ for $j=2,\dots,\ell$.  Consequently, $\exp\{\gamma {\balpha^{(1)}}' \bu\}$ dominates $\exp\{\gamma {\balpha^{(j)}}' \bu\}$ for $j=2,\dots,\ell$ as $\gamma \rightarrow \infty$.  Based on the nonnegativity of $f(\bx)$, we conclude that $\bc_1 > 0$.  Thus, a coefficient $\bc_j$ must be positive if the corresponding exponent $\balpha^{(j)}$ is an extremal exponent.

These observations imply that the cone of coefficients of AM/GM-exponentials $\mathcal{C}_\age(\balpha^{(1)},\dots,\allowbreak\balpha^{(\ell)};\balpha^{(0)})= \R^{\ell+1}_{+}$ if $\balpha^{(0)} \allowbreak \notin \allowbreak \mathrm{conv}(\{\balpha^{(j)}\}_{j=1}^\ell$; this follows from the point that AM/GM-exponentials \emph{exactly} correspond to nonnegative signomials with at most one negative coefficient.  Consequently, one can exclude such posynomials in the description of SAGE functions in \eqref{eq:sagefinite}, which also leads to simpler representations in Proposition~\ref{prop:sagecone}.  These modifications may provide computational benefits if one wishes to certify nonnegativity of a signomial with the property that many of the exponents are extremal exponents.


\section{Hierarchies of Relative Entropy Relaxations for SPs} \label{sec:hierarchy}

In this section we describe hierarchies of convex relaxations to obtain lower bounds for SPs.  Our development parallels the literature on convex relaxations for polynomial optimization \cite{Las2001,Par2000,Par2003}, but with an important distinction -- our approach for SPs is based on relative entropy optimization derived via SAGE decompositions, while relaxations for polynomial optimization are based on SDPs derived via sum-of-squares decompositions.  Relative entropy relaxations are better suited than SDP relaxations for SPs, and we comment on these points in greater detail in Section~\ref{subsec:polyopt}.  We also give theoretical support for our hierarchies in Section~\ref{subsec:completeness} by showing that they provide convergent sequences of lower bounds to the optimal value for broad classes of SPs.  Finally, we present a dual perspective of our methods, which leads to a technique for recognizing when our lower bounds are tight and for extracting optimal solutions.

\subsection{A Hierarchy for Unconstrained SPs} \label{subsec:unconstrained}

Globally minimizing a signomial $f(\bx)$ is equivalent to maximizing a lower bound of $f(\bx)$:
\begin{equation}
f_\star ~~=~~ \inf_{\bx \in \R^n} ~ f(\bx) ~~=~~ \sup_{\gamma \in \R} ~ \gamma ~\mathrm{s.t.} ~ f(\bx) - \gamma \geq 0 ~ \forall \bx \in \R^n. \label{eq:generalspunc}
\end{equation}
Replacing the global nonnegativity condition on $f(\bx) - \gamma$ by a sufficient condition for nonnegativity gives lower bounds on the optimal value $f_{\star}$.  SAGE functions provide a natural sufficient condition for global nonnegativity leading to a computationally tractable approach for obtaining lower bounds, and we describe this relaxation next.

To state things concretely, fix a set of exponentials $\{\balpha^{(j)}\}_{j=1}^\ell \subset \R^n$ with $\balpha^{(1)} = \bzero$, and consider a signomial $f(\bx) = \sum_{j=1}^\ell \bc_j \exp\{{\balpha^{(j)}}' \bx\}$ defined with respect to these exponents. (If $f$ contains no constant term, then set the corresponding coefficient $\bc_1 = 0$.)  In order to produce a lower bound on the global minimum $f_\star$ of $f$, we consider the following convex relaxation based on SAGE decompositions:
\begin{equation}
f_{\sage} = \sup_{\gamma \in \R} ~ \gamma ~~~ \mathrm{s.t.} ~ f(\bx) - \gamma \in \sage\left(\{\balpha^{(j)}\}_{j=1}^\ell\right). \label{eq:sagegenerallowerbound}
\end{equation}
Based on the preceding discussion, it is clear that $f_{\sage} \leq f_\star$.  Furthermore, computing $f_{\sage}$ can be reformulated as follows via Proposition~\ref{prop:sagecone}:
\begin{equation}
f_{\sage} = \sup_{\gamma \in \R} ~ \gamma ~~~ \mathrm{s.t.} ~ (\bc_1 -\gamma, \bc_2, \dots, \bc_\ell) \in \mathcal{C}_\sage(\balpha^{(1)}, \dots, \balpha^{(\ell)}), \label{eq:sagelowerbound}
\end{equation}
where $\mathcal{C}_\sage(\balpha^{(1)},\dots,\balpha^{(\ell)})$ is specified in \eqref{eq:sagecoeff2}.  Thus, the bound $f_\sage$ is the optimal value of a tractable convex program with linear and joint relative entropy constraints.

If $f(\bx)$ is a posynomial then $f_\sage = f_\star$; this follows from the fact that $f(\bx) - \gamma$ is a SAGE function if and only if $f(\bx) - \gamma \geq 0$ for all $\bx \in \R^n$ because $f(\bx) - \gamma$ has all but one coefficient being positive.  Consequently, the SAGE relaxation is exact for unconstrained GPs, which is reassuring as tractable problem instances (GPs are convex optimization problems that be solved efficiently) can be computed effectively in our framework.  More generally, we illustrate the utility of SAGE relaxations for minimizing signomials with multiple negative coefficients in the following examples.

\begin{example} \label{example:unc}
We consider signomials in three variables with seven terms of the form $f(\bx) = \sum_{j=1}^7 \bc_j \exp\{{\balpha^{(j)}}' \bx\}$ with $\bx \in \R^3$ and the following parameters fixed: $\bc_1=0$, $\bc_2=\bc_3=\bc_4 = 10$, $\balpha^{(1)} = (0,0,0)'$, $\balpha^{(2)} = (10.2,0,0)'$, $\balpha^{(3)} = (0,9.8,0)'$, and $\balpha^{(4)} = (0,0,8.2)'$.  The exponents $\balpha^{(5)},\balpha^{(6)},\balpha^{(7)} \in \R^3$ are chosen to be random vectors with entries distributed uniformly in $[0,3]$, and the coefficients $\bc_5,\bc_6,\bc_7$ are chosen to be random Gaussians with mean $0$ and standard deviation $10$.  This construction --- relatively large exponents $\balpha^{(2)},\balpha^{(3)},\balpha^{(4)}$ in comparison to $\balpha^{(5)},\balpha^{(6)},\balpha^{(7)}$, and the corresponding positive coefficients $\bc_2 = \bc_3 = \bc_4 = 10$ --- is an attempt to obtain signomials that are bounded below.  An example of a signomial generated in this manner is:
\small
\begin{equation}
\begin{aligned}
f(\bx) =& 10 \exp\{10.2 \bx_1\} + 10 \exp\{9.8 \bx_2\} + 10 \exp\{8.2 \bx_3\} \\ & - 14.6794 \exp\{1.5089\bx_1 +  1.0981 \bx_2 + 1.3419 \bx_3\} \\ & - 7.8601 \exp\{1.0857\bx_1 +  1.9069 \bx_2 + 1.6192 \bx_3\} \\ & + 8.7838 \exp\{1.0459 \bx_1 + 0.0492 \bx_2 + 1.6245 \bx_3\}.
\end{aligned}\label{eq:sig1}
\end{equation}
\normalsize
The SAGE relaxation \eqref{eq:sagelowerbound} applied to $f(\bx)$ gives the lower bound $f_\sage \approx -0.9747$.  By applying a technique presented in Section~\ref{subsec:dual}, we obtain the point $\bx^\star = (-0.3020, \allowbreak -0.2586, -0.4010)'$ with $f(\bx^\star) \approx -0.9747$.  Consequently, the lower bound $f_\sage$ is tight in this case.  As another instance of the construction described here, consider the randomly generated signomial:
\small
\begin{equation}
\begin{aligned}
f(\bx) =& 10 \exp\{10.2 \bx_1\} + 10 \exp\{9.8 \bx_2\} + 10 \exp\{8.2 \bx_3\} \\ & + 7.5907 \exp\{1.9864 \bx_1 + 0.2010 \bx_2 + 1.0855 \bx_3\} \\ & - 10.9888 \exp\{2.8242 \bx_1 +  1.9355 \bx_2 + 2.0503 \bx_3\} \\ &  -13.9164 \exp\{0.1828 \bx_1 + 2.7772 \bx_2 + 1.9001 \bx_3\}.
\end{aligned} \label{eq:sig2}
\end{equation}
\normalsize
In this case the SAGE relaxation \eqref{eq:sagelowerbound} is not tight: the SAGE lower bound is $-1.426$, and the technique from Section~\ref{subsec:dual} gives the upper bound $-0.739$.  More generally, we generated $80$ random signomials according to the above description, and the SAGE relaxation was tight in $63\%$ of the cases, while there was a gap in the remaining cases.
\end{example}

These examples demonstrate that there are several cases in which the lower bound $f_\sage$ equals the global minimum $f_\star$.  There are also situations in which the lower bound $f_\sage$ does not equal $f_\star$, which is not unexpected as global minimization of signomials is in general a computationally intractable problem.  Motivated by such cases, we describe next a methodology for obtaining a sequence of increasingly tighter lower bounds based on improved sufficient conditions for the global nonnegativity of the signomial $f(\bx)-\gamma$.  These improved lower bounds require solution of successively larger-sized relative entropy optimization problems.

Specifically, for a set of exponents $\{\balpha^{(j)}\}_{j=1}^\ell \subset \R^n$ with $\balpha^{(1)} = \bzero$ and a signomial $f(\bx) = \sum_{j=1}^\ell \bc_j \exp\{{\balpha^{(j)}}' \bx\}$ defined with respect to these exponents, we replace the SAGE condition in \eqref{eq:sagegenerallowerbound} by a stronger sufficient condition for the global nonnegativity of $f(\bx) - \gamma$, leading to the following lower bound for each $p \in \mathbb{Z}_+$:
\begin{equation}
f_\sage^{(p)} = \sup_{\gamma \in \R} ~ \gamma ~~~ \mathrm{s.t.} ~ \Big(\sum_{j=1}^\ell \exp\{{\balpha^{(j)}}' \bx\} \Big)^p [f(\bx) - \gamma] \in \sage(E_{p+1}(\{\balpha^{(j)}\}_{j=1}^\ell)). \label{eq:sagelowerboundp}
\end{equation}
The multiplier term $\left(\sum_{j=1}^\ell \exp\{{\balpha^{(j)}}' \bx\} \right)^p$ is globally nonnegative for all $\bx \in \R^n$, and consequently the constraint above is a sufficient condition for the global nonnegativity of $f(\bx) - \gamma$.  This in turn implies that $f_\sage^{(p)} \leq f_\star$ for all $p \in \mathbb{Z}_+$.  One can check that the bound $f_\sage$ of \eqref{eq:sagelowerbound} is equal to $f_\sage^{(0)}$ in \eqref{eq:sagelowerboundp}.  Computing $f_\sage^{(p)}$ for each $p$ can be recast as a relative entropy optimization problem by appealing to Proposition~\ref{prop:sagecone} and to the fact that the coefficients of the signomial $\left(\sum_{j=1}^\ell \exp\{{\balpha^{(j)}}' \bx\} \right)^p [f(\bx) - \gamma]$ are linear functionals of $\gamma$.  For example for the case $p=1$, the signomial $\left(\sum_{\tilde{j}=1}^\ell \exp\{{\balpha^{(\tilde{j})}}' \bx\} \right) [f(\bx) - \gamma] \in \sage(E_{2}(\{\balpha^{(j)}\}_{j=1}^\ell))$ is equivalent to $\sum_{j,\tilde{j}=1}^\ell \bc_{j,\tilde{j}} \exp\{[\balpha^{(j)}+\balpha^{(\tilde{j})}]' \bx\} \in \sage(E_{2}(\{\balpha^{(j)}\}_{j=1}^\ell))$, with $\bc_{j,\tilde{j}} = \bc_1 - \gamma$ whenever $j=1$ and $\bc_{j,\tilde{j}} = \bc_j$ otherwise; the latter SAGE condition can be reformulated using linear and relative entropy constraints on $\bc_{j,\tilde{j}}$, and in turn on $\gamma$, via Proposition~\ref{prop:sagecone}.  The next result states that $f_\sage^{(p)}$ is a non-decreasing function of $p$; this improvement carries a computational cost as the size of the associated relative entropy program grows with $p$.

\begin{lemma}  \label{lemma:uncmonotone}
Fix a collection of exponents $\{\balpha^{(j)}\}_{j=1}^\ell \subset \R^n$ with $\balpha^{(1)} = \bzero$, and consider a signomial $f(\bx)$ defined with respect to these exponents.  We have that $f_\sage^{(p)} \leq f_\sage^{(p+1)}$ for all $p \in \mathbb{Z}_+$.
\end{lemma}

\begin{proof}
For a fixed value of $\gamma$, let $g_p(\bx) = \left(\sum_{j=1}^\ell \exp\{{\balpha^{(j)}}' \bx\} \right)^p [f(\bx) - \gamma]$ and let $g_{p+1}(\bx) = \left(\sum_{j=1}^\ell \exp\{{\balpha^{(j)}}' \bx\} \right)^{p+1} [f(\bx) - \gamma]$.  We observe that:
\begin{equation*}
g_{p+1}(\bx) = \sum_{j=1}^\ell \exp\{{\balpha^{(j)}}' \bx\} g_{p}(\bx),
\end{equation*}
and that $g_p(\bx) \in \sage(E_{p+1}(\{\balpha^{(j)}\}_{j=1}^\ell))$ implies $\exp\{{\balpha^{(j)}}' \bx\} g_{p}(\bx) \in \sage(E_{p+2}(\{\balpha^{(j)}\}_{j=1}^\ell))$. Consequently, $g_p(\bx) \in \sage(E_{p+1}(\{\balpha^{(j)}\}_{j=1}^\ell))$ implies that $g_{p+1}(\bx) \in \sage(E_{p+2}(\{\balpha^{(j)}\}_{j=1}^\ell))$ as $\sage(E_{p+2}(\{\balpha^{(j)}\}_{j=1}^\ell))$ is closed under addition.  Since this implication is true for each fixed $\gamma$, we have the desired result.
\end{proof}

In Section~\ref{subsec:completeness}, we prove that the sequence $\{f_\sage^{(p)}\}$ converges to $f_\star$ as $p \rightarrow \infty$ for suitable families of signomials.  Consequently, the methodology described in this section leads to a convergent sequence of lower bounds, computed by solving increasingly larger relative entropy optimization problems.

\begin{example} \label{example:uncp}
Consider the signomial $f(\bx)$ from \eqref{eq:sig2} in Example~\ref{example:unc}.  The lower bound $f_\sage^{(0)} = -1.426$, and this is not tight.  The next level of the hierarchy presented above leads to the improved bound $f_\sage^{(1)} = -1.395$.
\end{example}

\subsection{From Unconstrained SPs to Constrained SPs: Algebraic Certificates of Compactness} \label{subsec:unc2con}
Our approach to developing relaxations for constrained SPs is qualitatively different from that in the unconstrained case.  Specifically, the hierarchy of relaxations that we develop for constrained SPs in Section~\ref{subsec:constrained} does not directly specialize to the hierarchy developed for unconstrained SPs in Section~\ref{subsec:unconstrained}.  The reason for this discrepancy is that our development in the constrained setting is based on the assumption that the constraint set is a compact set.  Concretely, fix a collection of exponents $\{\balpha^{(j)}\}_{j=1}^\ell \subset \R^n$, and let $f(\bx)$ and $C = \{g_i(\bx)\}_{i=1}^m$ be signomials with respect to these exponents.  Consider the constrained SP:
\begin{equation}
f_\star = \inf_{\bx \in \R^n} ~ f(\bx) ~~~ \mathrm {s.t.} ~ g_i(\bx) \geq 0, ~ i = 1,\dots, m. \label{eq:generalsp}
\end{equation}
The hierarchy we propose in Section~\ref{subsec:constrained} is shown to be complete in Section~\ref{subsec:completeness} (see Theorem~\ref{theorem:con}) under the premise that the constraint set $\K_C = \{\bx \in \R^n ~|~ g(\bx) \geq 0 ~ \forall g \in C\}$ is compact.  Crucially, the proof of completeness relies on the assumption that the compactness of $\K_C$ is explicitly enforced via inequalities of the form $U \geq \exp\{{\balpha^{(j)}}' \bx\} \geq L, ~ j=1,\dots,\ell$ (for appropriate $U,L \in \R_{++}$) in the list of constraints $C = \{g_i(\bx)\}_{i=1}^m$.  Such explicit enforcements of compactness enable us to appeal to representation theorems from real algebraic geometry that require appropriate Archimedeanity assumptions (see the survey \cite{Mar2008}), a point that is also significant in the development of hierarchies of SDP relaxations based on sum-of-squares decompositions for polynomial optimization problems \cite{Las2001,Par2000,Par2003,Put1993,Sch1991}.

As such a natural ``unconstrained'' counterpart of the constrained case is a setting in which we wish to minimize a signomial subject to a constraint set of the form $\{\bx \in \R^n ~|~ U \geq \exp\{{\balpha^{(j)}}' \bx\} \geq L, ~ j=1,\dots,\ell\}$ for some fixed $U,L \in \R_{++}$.  Alternatively, if we have some additional knowledge that a global minimizer of a signomial belongs to a set of this form for some known $U,L$, then the methods developed in Section~\ref{subsec:constrained} provide an alternative hierarchy of convergent lower bounds to those in Section~\ref{subsec:unconstrained}.

\subsection{A Hierarchy for Constrained SPs} \label{subsec:constrained}

Here we describe a hierarchy of relative entropy relaxations based on SAGE decompositions for constrained SPs.  In analogy with the unconstrained case, we can recast \eqref{eq:generalsp} as:
\begin{equation}
f_\star = \sup_{\gamma \in \R} ~ \gamma ~~~ \mathrm {s.t.} ~ f(\bx) - \gamma \geq 0 ~ \forall \bx \in \K_C. \label{eq:generalspnng}
\end{equation}
As the original constrained SP is computationally intractable to solve in general, certifying nonnegativity of a signomial over a constraint set defined by signomials is also intractable.  In order to obtain lower bounds on $f_\star$ that are tractable to compute, our approach is again to proceed by employing efficiently computable sufficient conditions for certifying nonnegativity of a signomial $f(\bx) - \gamma$ over a constraint set $\K_C$ defined by signomial inequalities.

A basic approach based on weak duality for bounding \eqref{eq:generalspnng} is to replace the constraint by the following sufficient condition for nonnegativity:
\begin{equation*}
f_{WD} = \sup_{\gamma \in \R, \bmu \in \R^m_+} ~ \gamma ~~~ \mathrm{s.t.} ~ f(\bx) - \gamma - \sum_{i=1}^m \bmu_i g_i(\bx) \geq 0 ~ \forall \bx \in \R^n.
\end{equation*}
As the nonnegativity condition here is a sufficient condition for the constraint in \eqref{eq:generalspnng}, we note that $f_{WD} \leq f_\star$; this is the standard argument underlying weak duality.  However, there are two difficulties with this approach.  The first is that the constraint still involves the computationally intractable problem of checking nonnegativity of an arbitrary signomial.  The second difficulty, independent of the first one, is that the underlying SP is nonconvex, and therefore we do not expect strong duality to hold; as such, even if one could compute $f_{WD}$ it is in general the case that $f_{WD} < f_\star$.  To remedy the first difficulty, we replace the nonnegativity condition by one based on SAGE decomposability.  To address the second challenge we add valid inequalities that are consequences of the original set of constraints; although this idea of adding appropriate redundant constraints is an old one, it has been employed to particularly powerful effect in polynomial optimization problems \cite{Las2001,Par2000,Par2003}.

Formally, consider the set of signomials defined as products of the original set of constraint functions $C = \{g_i(\bx)\}_{i=1}^m$ in \eqref{eq:generalspnng}:
\begin{equation}
R_q(C) = \left\{\prod_{k=1}^q h_k ~|~ h_k \in \{1\} \cup C \right\}. \label{eq:validineq}
\end{equation}
Here $1$ represents the signomial that is identically equal to one for all $\bx \in \R^n$.  Based on this definition, we consider the following relaxation of \eqref{eq:generalspnng}:
\begin{equation}
\begin{aligned}
f_\sage^{(p,q)} = & \sup_{\substack{\gamma \in \R, \\ s_h(\bx) \in \sage(E_p(\{\balpha^{(j)}\}_{j=1}^\ell))}} ~ \hspace{1in}\gamma \\ & \hspace{0.65in}\mathrm{s.t.} \hspace{0.75in} f(\bx) - \gamma - \sum_{h(\bx) \in R_q(C)} s_h(\bx) h(\bx)  \\ & \hspace{2.4in} \in \sage(E_{p+q}(\{\balpha^{(j)}\}_{j=1}^\ell)).
\end{aligned} \label{eq:sagelowerboundpq}
\end{equation}
For a fixed $q$, the number of terms in the set $R_q(C)$ is at most $(\ell+1)^q$, and consequently the sum consists of a finite number of terms.  Moreover, computing $f^{(p,q)}_\sage$ can be recast as a relative entropy optimization problem via Proposition~\ref{prop:sagecone} as the coefficients of the signomial $f(\bx) - \gamma - \sum_{h \in R_q(C)} s_h(\bx) h(\bx)$ are linear functionals of $\gamma$ and the coefficients of the SAGE functions $s_h(\bx)$.  Note that for $h \in R_q(C)$ and $s_h(\bx) \in \sage(E_p)$, the signomial $s_h(\bx) h(\bx)$ is nonnegative on the set $\K_C$, thus corresponding to a valid constraint.  Hence, for a fixed value of $\gamma$ and fixed SAGE functions $s_h(\bx)$, the condition $f(\bx) - \gamma - \sum_{h \in R_q(C)} s_h(\bx) h(\bx) \in \sage(E_{p+q})$ is a sufficient condition for the nonnegativity of $f(\bx) - \gamma$ on the constraint set $\K_C$.  Consequently we have that $f^{(p,q)}_\sage \leq f_\star$ for all $p,q \in \mathbb{Z}_+$.

The next result records the fact that $f^{(p,q)}_\sage$ provides increasingly tighter bounds as $p,q$ grow larger.  This improvement comes at the expense of solving larger relative entropy optimization problems for bigger values of $p,q$.
\begin{lemma} \label{lemma:conmonotone}
Fix a set of exponents $\{\balpha^{(j)}\}_{j=1}^\ell \subset \R^n$ with $\balpha^{(1)} = \bzero$, and let $f(\bx)$ and $C = \{g_i(\bx)\}_{i=1}^m$ be signomials with respect to these exponents.  Then we have that $f_\sage^{(p,q)} \leq f_\sage^{(p',q')}$ for $p \leq p', q \leq q'$ with $p,p',q,q' \in \mathbb{Z}_+$.
\end{lemma}
\begin{proof}
This result follows from the observations that $\sage(E_p) \subset \sage(E_{p'})$, $R_q(C) \subset R_{q'}(C)$, and $\sage(E_{p+q}) \subset \sage(E_{p'+q'})$ if $p \leq p', q \leq q'$.
\end{proof}

In Section~\ref{subsec:completeness} we show that for SPs with compact constraint sets, $f^{(p,q)}_\sage \rightarrow f_\star$ as $p,q  \rightarrow \infty$.  These results demonstrate the value in the redundant constraints employed in the bounds $f_\sage^{(p,q)}$.  We also note that GPs are solved exactly at the first level of the hierarchy, i.e., the lower bound $f_\sage^{(0,1)}$ equals the global minimum for GPs.

\begin{example}
As our first example, consider a constrained SP in which the objective to be minimized is the signomial \eqref{eq:sig1} from Example~\ref{example:unc}, and the constraint set is convex as follows:
\small
\begin{eqnarray*}
\{\bx \in \R^3 ~|~&& 8 \exp\{10.2 \bx_1\} +8 \exp\{9.8 \bx_2\} + 8 \exp\{8.2 \bx_3\} \\ && +6.4 \exp\{1.0857\bx_1 +  1.9069 \bx_2 + 1.6192 \bx_3\} \leq 1\}.
\end{eqnarray*}
\normalsize
Notice that the exponents in $f(\bx)$ in \eqref{eq:sig1} and ones in the constraint here are common.  Although this constraint set is a convex set, this problem is not a GP because the objective function is a signomial consisting of two negative coefficients.  Applying the first level of the SAGE hierarchy described in this subsection gives the lower bound $f^{(0,1)}_\sage = -0.6147$.  This bound is tight and it is achieved at the feasible point $\bx^\star = (-0.4312, -0.3823, -0.6504)'$.
\end{example}

\begin{example}
As our next example, consider a constrained SP in which the objective to be minimized is again the signomial \eqref{eq:sig1} from Example~\ref{example:unc}, and the constraint is given by following signomial:
\small
\begin{equation}
\begin{aligned}
g(\bx) =& -8 \exp\{10.2 \bx_1\} -8 \exp\{9.8 \bx_2\} - 8 \exp\{8.2 \bx_3\} \\ & + 0.7410 \exp\{1.5089\bx_1 +  1.0981 \bx_2 + 1.3419 \bx_3\} \\ & - 0.4492 \exp\{1.0857\bx_1 +  1.9069 \bx_2 + 1.6192 \bx_3\} \\ & + 1.4240 \exp\{1.0459 \bx_1 + 0.0492 \bx_2 + 1.6245 \bx_3\}.
\end{aligned}\label{eq:sigcon1}
\end{equation}
\normalsize
This constraint was chosen in a similar manner to the signomials in Example~\ref{example:unc}.  Keeping the exponents fixed to the same values as those in \eqref{eq:sig1}, the first three coefficients are chosen to be equal to $-8$ and the last three are chosen to be random Gaussians with mean $0$ and standard deviation $1$.
Thus, we wish to minimize $f(\bx)$ from \eqref{eq:sig1} subject to the constraint that $g(\bx) \geq 0$.  The first level of the hierarchy for constrained SPs gives us the lower bound $f^{(0,1)}_\sage = -0.7372$.  This bound is also tight and it is achieved at $\bx^\star  = (0.0073, 0.0065, 0.0130)'$.
\end{example}

\subsection{Dual Perspectives and Extracting Optimal Solutions} \label{subsec:dual}
The hierarchies of relaxations described in the previous subsections have an appealing dual perspective.  Focussing on the unconstrained case for simplicity, the problem of minimizing a signomial $f(\bx) = \sum_{j=1}^\ell \bc_j \exp\{ {\balpha^{(j)}}' \bx\}$ with $\balpha^{(1)} = \bzero$ can be recast as:
\begin{equation*}
f_\star = \inf_{\bv \in \R^\ell} ~ \bc' \bv ~~~ \mathrm{s.t.} ~ \bv \in \{(1, \dots, \exp\{{\balpha^{(\ell)}}' \bx \}) ~|~ \bx \in \R^n\}.
\end{equation*}
As the objective is linear, this problem can equivalently be expressed as a convex program by convexifying the constraint set to $\mathrm{conv}\{(1, \dots, \exp\{{\balpha^{(\ell)}}' \bx \}) ~|~ \bx \in \R^n\}$.  However, checking membership in $\mathrm{conv}\{(1, \dots, \exp\{{\balpha^{(\ell)}}' \bx \}) ~|~ \bx \in \R^n\}$ is in general an intractable problem.

The methods described in Section~\ref{subsec:unconstrained} can be viewed as providing a sequence of tractable \emph{outer convex approximations} to the set $\{(1, \dots, \exp\{{\balpha^{(\ell)}}' \bx \}) ~|~ \bx \in \R^n\}$.  Specifically, the dual corresponding to the SAGE lower bound $f^{(p)}_\sage$ \eqref{eq:sagelowerboundp} is:
\begin{equation}
f^{(p)}_\sage = \inf_{\bv \in \R^\ell} ~ \bc' \bv ~~~ \mathrm{s.t.} ~ \bv \in \mathcal{W}_p(\bzero,\dots,\balpha^{(\ell)}), \label{eq:sagepdual}
\end{equation}
where each $\mathcal{W}_p(\bzero,\dots,\balpha^{(\ell)}) \subset \R^\ell$ is a convex set that contains $\{(1, \dots, \exp\{{\balpha^{(\ell)}}' \bx \}) ~|~ \allowbreak \bx \in \R^n\}$.  The set $\mathcal{W}_0(\bzero,\dots,\balpha^{(\ell)})$ corresponding to $f^{(0)}_\sage$ \eqref{eq:sagelowerboundp} is:
\begin{equation}
\mathcal{W}_0(\bzero,\dots,\balpha^{(\ell)}) = \Big\{\bv \in \R^\ell_+ ~|~ \bv_1 = 1 ~\mathrm{and}~ \bv \in \mathcal{C}^\star_\sage(\bzero,\dots,\balpha^{(\ell)})\Big\}, \label{eq:dualW0}
\end{equation}
where $\mathcal{C}^\star_\sage(\bzero,\dots,\balpha^{(\ell)})$ is characterized in \eqref{eq:sagedualcone}.  The next result shows that $\mathcal{W}_0(\bzero,\allowbreak\dots,\balpha^{(\ell)})$ is an outer approximation of the set $\{(1, \dots, \exp\{{\balpha^{(\ell)}}' \bx \}) ~|~ \bx \in \R^n\}$:
\begin{lemma} \label{lemma:expcurve}
Fix a set of exponents $\{\balpha^{(j)}\}_{j=1}^\ell \subset \R^n$ with $\balpha^{(1)} = \bzero$, and consider the convex set $\mathcal{W}_0(\bzero,\dots,\balpha^{(\ell)})$ described in \eqref{eq:dualW0}.  Then we have that $\{(1, \dots, \allowbreak \exp\{{\balpha^{(\ell)}}' \bx \}) ~|~ \bx \in \R^n\} \subseteq \mathcal{W}_0(\bzero,\dots,\balpha^{(\ell)})$.
\end{lemma}
\begin{proof}
For each $\bx \in \R^n$, we note that $\bv = (1,\dots,\exp\{{\balpha^{(\ell)}}'\bx \})' \in \mathcal{W}_0(\bzero,\dots,\balpha^{(\ell)})$ by setting $\btau^{(i)} = \exp\{ {\balpha^{(i)}}' \bx\} ~ \bx$ in the characterization \eqref{eq:sagedualcone} of $\mathcal{C}^\star_\sage(\bzero,\allowbreak \dots,\allowbreak \balpha^{(\ell)})$.
\end{proof}

This lemma gives an alternative justification via duality that $f_\sage^{(0)} \leq f_\star$.  The higher levels of the hierarchy in Section~\ref{subsec:unconstrained} corresponding to the improved bounds $f_\sage^{(p)}$ provide tighter convex approximations of the set $\{(1, \dots, \allowbreak \exp\{{\balpha^{(\ell)}}' \bx \}) ~|~ \allowbreak \bx \in \R^n\}$.  For example, one can check that the set $\mathcal{W}_1(\bzero,\dots,\balpha^{(\ell)})$ is given as follows:
\begin{eqnarray*}
\mathcal{W}_1(\bzero,\dots,\balpha^{(\ell)}) &= \Big\{\bv \in \R^\ell_+ ~|~ & \bv_1 = 1 ~\mathrm{and}~ \exists \blambda^{(i)} \in \R^\ell, i =1,\dots,\ell, ~\mathrm{s.t.} \sum_{i=1}^\ell \blambda^{(i)} = \bv, \\ && [{\blambda^{(1)}}',\dots,{\blambda^{(\ell)}}']' \in \mathcal{C}_\sage^\star\left(\bzeta^{(1,1)}, \dots, \bzeta^{(\ell,\ell)} \right)\Big\},
\end{eqnarray*}
where $\bzeta^{(i,j)} = \balpha^{(i)} + \balpha^{(j)}$ for $i,j=1,\dots,\ell$.


The dual perspective presented here also provides a method to check whether the lower bound $f_\sage^{(p)}$ is tight, i.e., $f^{(p)}_\sage = f_\star$.  Focussing on the first level of the hierarchy with $p=0$, suppose an optimal solution $\hat{\bv} \in \R^\ell_+$ of the convex program \eqref{eq:sagepdual} is such that there exists a corresponding point $\hat{\bx} \in \R^n$ so that $\hat{\bv}_j = \exp\{{\balpha^{(j)}}' \hat{\bx} \}$ for each $j = 1,\dots,\ell$.  Then it follows that $f_\sage^{(0)} = f_\star$ because the optimal solution $\hat{\bv}$ belongs to the set $\{(1, \dots, \exp\{{\balpha^{(\ell)}}' \bx \}) ~|~ \bx \in \R^n\}$, and the point $\hat{\bx} \in \R^n$ is the optimal solution of the original SP \eqref{eq:generalspunc}.  If $\hat{\bv} \notin \{(1, \dots, \exp\{{\balpha^{(\ell)}}' \bx \}) ~|~ \bx \in \R^n\}$, then a natural heuristic to attempt to obtain a good approximation of a minimizer of the original SP is to project $\log(\hat{\bv})$ onto the image of the $\ell \times n$ matrix formed by setting the rows to be the $\balpha^{(j)}$'s.  Denoting this projection by $\bdelta \in \R^\ell$, we obtain an upper bound of the optimal value of the SP \eqref{eq:generalspunc} by computing $\sum_{j=1}^\ell \bc_j \exp\{\bdelta_j\}$.  This is the technique employed in Example~\ref{example:unc} to produce upper bounds.

The dual perspective in the constrained case is conceptually similar.  Fix a set of exponents $\{\balpha^{(j)}\}_{j=1}^\ell \subset \R^n$ with $\balpha^{(1)} = \bzero$, let $f(\bx) = \sum_{j=1}^\ell \bc^{(f)}_j \exp\{{\balpha^{(j)}}' \bx\}$ be the objective, and let $C = \{g_i(\bx)\}_{i=1}^m$ specify the set of signomials defining the constraints with $g_i(\bx) = \sum_{j=1}^\ell \bc^{(g_i)}_j \exp\{{\balpha^{(j)}}' \bx\}$. One can recast the constrained SP \eqref{eq:generalsp} as:
\begin{eqnarray*}
f_\star = \inf_{\bv \in \R^\ell_+} ~ {\bc^{(f)}}' \bv ~ \mathrm{s.t.} ~ \bv \in \{(1, \dots, \exp\{{\balpha^{(\ell)}}' \bx \}) ~|~ \bx \in \R^n\}, ~ {\bc^{(g_i)}}' \bv \geq 0, ~ i=1,\dots,m.
\end{eqnarray*}
As the objective function is linear, the constraint set can be replaced by its convex hull. By considering the duals of the convex programs \eqref{eq:sagelowerboundpq} corresponding to $f_\sage^{(p,q)}$, one can again interpret SAGE relaxations as providing a sequence of computationally tractable convex outer approximations of the above constraint set:
\begin{equation}
f^{(p,q)}_\sage = \inf_{\bv \in \R^\ell} ~ {\bc^{(f)}}' \bv ~~~ \mathrm{s.t.} ~ \bv \in \mathcal{W}_{p,q}(\bzero,\dots,\balpha^{(\ell)}; \bc^{(g_1)}, \dots, \bc^{(g_m)}). \label{eq:sagepqdual}
\end{equation}
That is, $\mathcal{W}_{p,q}(\bzero,\dots,\balpha^{(\ell)}; \bc^{(g_1)}, \dots, \bc^{(g_m)})$ can be viewed as a convex outer approximation of the set $\{(1, \dots, \exp\{{\balpha^{(\ell)}}' \bx \}) ~|~ \bx \in \R^n\} \cap \{\bv ~|~ {\bc^{(g_i)}}' \bv \geq 0, ~ i=1,\dots,m\}$.  For example, the first level of the hierarchy corresponding to $f^{(0,1)}$ is given by:
\begin{eqnarray*}
\mathcal{W}_{0,1}(\bzero,\dots,\balpha^{(\ell)}; \bc^{(g_1)}, \dots, \bc^{(g_m)}) &= \Big\{\bv \in \R^\ell_+ ~|~ & \bv_1 = 1, ~ \bv \in \mathcal{C}^\star_\sage(\bzero,\dots,\balpha^{(\ell)}), \\ && \bv' \bc^{(g_i)} \geq 0, ~ i=1,\dots,m \Big\}.
\end{eqnarray*}
Analogous to Lemma~\ref{lemma:expcurve}, one can show that $\mathcal{W}_{0,1}(\bzero,\dots,\balpha^{(\ell)}; \bc^{(g_1)}, \dots, \bc^{(g_m)})$ is indeed a convex outer approximation of the set $\{(1, \dots, \exp\{{\balpha^{(\ell)}}' \bx \}) ~|~ \bx \in \R^n\} \allowbreak \cap \{\bv ~|~ {\bc^{(g_i)}}' \bv \geq 0, ~ i=1,\dots,m\}$.  As with the unconstrained setting, one can check whether $f^{(p,q)}_\sage = f_\star$ by verifying that the optimal solution of the dual problem \eqref{eq:sagepqdual} lies in the set $\{(1, \dots, \exp\{{\balpha^{(\ell)}}' \bx \}) ~|~ \bx \in \R^n\} \allowbreak \cap \{\bv ~|~ {\bc^{(g_i)}}' \bv \geq 0, ~ i=1,\dots,m\}$.


\section{Completeness of Hierarchies} \label{sec:completeness}
In this section, we show that the hierarchies of lower bounds presented in the previous section converge to the optimal value for suitable classes of SPs.  We also contrast the methods developed in this paper with those based on SDPs for polynomial optimization problems \cite{Las2001,Par2000,Par2003}, highlighting the point that relative entropy relaxations based on SAGE decompositions are better suited to SPs than SDP relaxations based on sum-of-squares techniques.

\subsection{Completeness of Relative Entropy Hierarchies for SPs with Rational Exponents} \label{subsec:completeness}

Our completeness results are stated for SPs in which the signomials consist of rational exponent vectors.  Specifically, the high-level idea behind our proofs is as follows: We transform the SPs to appropriate polynomial optimization problems over the nonnegative orthant by suitably clearing denominators in the exponents of the SP, then apply Positivstellensatz results from real algebraic geometry to obtain representations of positive polynomials \cite{Kri1964,Mar2008,Ste1974} in terms of SAGE functions (which can be viewed as nonnegative polynomials over the orthant modulo an ideal), and finally transform the results back to derive a consequence in terms of signomials.  We emphasize that the transformation of an SP to a polynomial optimization problem by clearing denominators in the exponents is purely an element of our proof technique due to the appeal to representation theorems from real algebraic geometry.  Our algorithmic methodology described in Section~\ref{sec:hierarchy} does not require such clearing of denominators, and we highlight the significance of this point in Section~\ref{subsec:polyopt}.

\paragraph{Unconstrained SPs} Our transformation of SPs with rational exponents leads to polynomial optimization problems over the positive orthant intersected with an algebraic variety.  However, our appeal to Positivstellensatz results from real algebraic geometry gives representations of polynomials over the \emph{nonnegative} orthant intersected with an algebraic variety.  Further, these representation theorems also entail certain Archimedeanity conditions as discussed in Appendix~\ref{app:rep}.  In order to overcome both these issues, we require an additional condition on the exponents as elaborated in the following result.

\begin{theorem} \label{theorem:unc}
Let $\{\balpha^{(j)}\}_{j=1}^\ell \subset \Q^n$ be a collection of rational exponents with the following properties: $(a)$ the first $n$ exponents $\{\balpha^{(j)}\}_{j=1}^n$ are linearly independent, $(b)$ $\balpha^{(n+1)} = \bzero$, and $(c)$ the remaining exponents $\{\balpha^{(j)}\}_{j=n+2}^\ell$ lie in the convex hull $\mathrm{conv}(\{\balpha^{(i)}\}_{i=1}^{n+1})$, and $\balpha^{(j)} \notin \mathrm{conv}(\{\balpha^{(i)}\}_{i=1}^{n})$ for $j=n+2,\dots,\ell$.  Suppose $f(\bx)$ is a signomial with respect to these exponents such that $\inf_{\bx \in \R^n} ~ f(\bx) > 0$ and that $\bc_j > 0$ for each $j \in \{1,\dots,n\}$.  Then there exists $p \in \mathbb{Z}_+$ such that $\left(\sum_{j=1}^\ell \exp\{{\balpha^{(j)}}' \bx\} \right)^p f(\bx) \in \sage(E_{p+1}(\{\balpha^{(j)}\}_{j=1}^\ell))$.
\end{theorem}

\paragraph{Note} The condition that $\bc_j > 0$ for each $j = 1,\dots,n$ is not restrictive because nonzero coefficients corresponding to extremal exponents must be positive for any globally nonnegative signomial; indeed, if any of these $\bc_j$'s for $j=1,\dots,n$ is negative then the signomial $f(\bx)$ is unbounded below (see Section~\ref{subsec:nngsig}).

This theorem is similar in spirit to Polya-Reznick type results on rational sum-of-squares representations of positive polynomials with uniform denominators \cite{Pol1928,Rez1995}.  Without loss of generality one can take the first $n$ exponent vectors $\{\balpha^{(j)}\}_{j=1}^\ell$ to be linearly independent to satisfy condition $(a)$; generalizing our discussion to address the case in which there are fewer than $n$ linearly independent vectors in $\{\balpha^{(j)}\}_{j=1}^\ell$ is straightforward.  Further, one can also add the zero vector to the collection of exponents $\{\balpha^{(j)}\}_{j=1}^\ell$ to satisfy condition $(b)$.  Condition $(c)$ is the key assumption in this theorem as it enables us to overcome the two difficulties described above preceding the statement of Theorem~\ref{theorem:unc} (see Appendix~\ref{app:uncproof}); note that one cannot simply add vectors to the collection of exponents to satisfy this condition because each of the coefficients $\bc_j$ for $j=1,\dots,n$ must be nonzero. Observe also that the structure of the exponents in the signomials considered in Example~\ref{example:unc} satisfy the conditions of Theorem~\ref{theorem:unc}.

To see that Theorem~\ref{theorem:unc} yields completeness for unconstrained SPs, suppose $f(\bx)$ is a signomial with respect to a set of exponents $\{\balpha^{(j)}\}_{j=1}^\ell$ satisfying the assumptions of Theorem~\ref{theorem:unc}, and let $f_\star$ denote the global minimum of $f(\bx)$.  For each $\epsilon < f_\star$, the signomial $f(\bx) - \epsilon$ is strictly positive for all $\bx \in \R^n$.  Theorem~\ref{theorem:unc} asserts that there exists a sufficiently large $p_\epsilon \in \mathbb{Z}_+$ such that $\left(\sum_{j=1}^\ell \exp\{{\balpha^{(j)}}' \bx\} \right)^{p_\epsilon} [f(\bx) - \epsilon] \in \sage(E_{p_\epsilon+1}(\{\balpha^{(j)}\}_{j=1}^\ell))$.  Combined with the monotonicity result of Lemma~\ref{lemma:uncmonotone}, one can conclude that $f^{(p)}_\sage \rightarrow f_\star$ as $p \rightarrow \infty$.

\paragraph{Constrained SPs} Our next result states that the hierarchy of lower bounds $f^{(p,q)}_\sage$ for constrained SPs also converges to the global optimum if the constraint set is compact.  Consider an SP specified in terms of an objective signomial $f(\bx)$ and constraint signomials $C = \{g_i(\bx)\}_{i=1}^m$, all defined with respect to a collection of exponents $\{ \balpha^{(j)}\}_{j=1}^\ell \subset \R^n$.  As discussed in Section~\ref{subsec:unc2con}, the compactness of the constraint set $\K_C = \{\bx ~|~ g_i(\bx) \geq 0, ~ i=1,\dots,m\}$ must be explicitly certified via the presence of redundant inequalities of the form $U \geq \exp\{{\balpha^{(j)}}' \bx\} \geq L, ~ j=1,\dots,\ell$ for suitable $U,L \in \R_{++}$ in the list of constraints $C$.  If $\K_C$ is a compact set, then such redundant constraints can always be added for appropriately large $U$ and small $L$.  (The compactness of the constraint set ensures that some of the difficulties in the unconstrained case -- the Archimedeanity requirement and the distinction between positive versus nonnegative orthant in appealing to representation theorems -- do not present obstacles in the constrained setting; specifically, we do not require additional conditions on the exponents beyond the assumption that they are rational vectors).

\begin{theorem} \label{theorem:con}
Fix a set of \emph{rational} exponents $\{\balpha^{(j)}\}_{j=1}^\ell \subset \Q^n$, and let $f(\bx)$ and $C = \{g_i(\bx) \}_{i=1}^m$ be signomials with respect to these exponents.  Let $R_q(C)$ be as defined in \eqref{eq:validineq}, and let the constraint set be $\K_C = \{\bx ~|~ g_i(\bx) \geq 0, ~ i=1,\dots,m\}$.  Suppose inequalities of the form $U \geq \exp\{{\balpha^{(j)}}'\bx\} \geq L, ~j=1,\dots,\ell$ for $U,L \in \R_{++}$ are formally specified in the list of constraints $C$ (explicitly serving as witnesses of the compactness of $\K_C$), and suppose $f(\bx)$ is strictly positive for all $\bx \in \K_C$.  Then there exist $p, q \in \mathbb{Z}_+$ and SAGE functions $s_h(\bx) \in \sage(E_p(\{\balpha^{(j)}\}_{j=1}^\ell))$ indexed by $h \in R_q(C)$ such that $f(\bx) - \sum_{h(\bx) \in R_q(C)} s_h(\bx) h(\bx) \in \sage(E_{p+q}(\{\balpha^{(j)}\}_{j=1}^\ell))$.
\end{theorem}

By proceeding with a line of reasoning as in the unconstrained case, one can check that the sequence of lower bounds $f_\sage^{(p,q)}$ converges to the global optimum of a constrained SP with rational exponents and a compact constraint set.  Note that Theorem~\ref{theorem:con} also implies that the methods described in Section~\ref{subsec:constrained} can be employed to find a minimizer of a signomial over a constraint set of the form $\{\bx \in \R^n ~|~ U \geq \exp\{{\balpha^{(j)}}'\bx\} \geq L, ~j=1,\dots,\ell \}$ for $U,L \in \R_{++}$ (or for unconstrained SPs when bounds on the location of a global minimizer are known).

\paragraph{Remarks} The proof of Theorem~\ref{theorem:unc} uses the fact that the set of posynomials forms a preprime (i.e., closed under addition and multiplication) and that posynomials are SAGE functions.  The proof of Theorem~\ref{theorem:con} appeals to the property that the set of SAGE functions is closed under multiplication by posynomials (formally, that the set of SAGE functions forms a module over the preprime of posynomials).  Although these results provide theoretical support for the methods proposed in Section~\ref{sec:hierarchy}, our proof techniques do not reflect the full strength of relative entropy relaxations.  For example, one of the appealing features of relative entropy relaxations for SPs -- not revealed by the proofs of Theorems~\ref{theorem:unc} and \ref{theorem:con} -- is that they are robust to small perturbations of the exponents of the underlying SP.  We discuss these points in Section~\ref{subsec:polyopt}; see Proposition~\ref{prop:robustness} and Example~\ref{example:robust}.

\subsection{Contrast with Polynomial Optimization} \label{subsec:polyopt}

A natural point of comparison with the framework presented in this paper is the literature on polynomial optimization, which involves the minimization of a multivariate polynomial subject to constraints specified by multivariate polynomials.  As with SPs, polynomial optimization problems are also non-convex in general, and they include families of NP-hard problems.  Parrilo \cite{Par2000,Par2003} and Lasserre \cite{Las2001} describe computationally feasible methods to obtain lower bounds for polynomial optimization problems.  These techniques rely on nonnegativity certificates for polynomials based on sum-of-squares decompositions \cite{Kri1964,Mar2008,Ste1974}, and the observation by Shor \cite{Sho1987} that checking if a polynomial is a sum-of-squares can be recast as an SDP feasibility problem.

SPs with rational exponents can be transformed to polynomial optimization problems over the nonnegative orthant by clearing denominators in the exponents.  This transformation generally leads to polynomials of very large degrees, thus making sum-of-squares techniques ill-suited for general SPs.  Example~\ref{example:basicage} gives a concrete illustration of this point.  The signomial $f(\bx) = \exp\{\bx_1\} + \exp\{\bx_2\} - \exp\{(\tfrac{1}{2}+\epsilon) \bx_1 +(\tfrac{1}{2}-\epsilon) \bx_2 \}$ for a fixed $\epsilon \in [-\tfrac{1}{2},\tfrac{1}{2}]$ can be efficiently certified as being globally nonnegative by solving a suitable relative entropy feasibility problem.  Now suppose that $\epsilon = \tfrac{p}{q}$ is a rational number in $[-\tfrac{1}{2},\tfrac{1}{2}]$.  One can then transform the question of nonnegativity of $f(\bx)$ to a problem of polynomial nonnegativity.  In particular by setting $\bz_1 = \exp\{\bx_1\}, \bz_2 = \exp\{\bx_2\}, \bz_3 = \exp\{(\tfrac{1}{2}+\tfrac{p}{q}) \bx_1 +(\tfrac{1}{2}-\tfrac{p}{q}) \bx_2 \}$, we would like to certify the nonnegativity of the polynomial $\tilde{f}(\bz_1,\bz_2,\bz_3) = \bz_1 + \bz_2 - \bz_3$ over the nonnegative orthant $\R^3_+$ modulo the ideal generated by the polynomial $\bz_3^{2q} - \bz_1^{q+2p} \bz_2^{q-2p}$.  If $q$ is large, the corresponding certificates based on sum-of-squares methods can be of very large degree (see \cite{Las2001,Par2000,Par2003} for more details about constructing these certificates), which in turn require solution of large SDPs.  On the other hand, there exists a short certificate for the nonnegativity of $f(\bx)$ via the AM/GM inequality, which is the basic insight we exploit to develop tractable convex relaxations for SPs.

More broadly, relative entropy relaxations for SPs also have the virtue that the bounds they provide are generally robust to small perturbations of the exponents.  Specifically, given a set of exponents $\{\balpha^{(j)}\}_{j=1}^\ell \subset \R^n$, we say that a SAGE cone $\mathcal{C}_\sage\left(\balpha^{(1)}, \dots, \balpha^{(\ell)} \right)$ is robust to small perturbations of the exponents $\{\balpha^{(j)}\}_{j=1}^\ell$ when the following condition holds: if a coefficient vector $\bc \in \R^\ell$ belongs to the interior of the cone $\mathcal{C}_\sage\left(\balpha^{(1)}, \dots, \balpha^{(\ell)} \right)$ for a collection of exponents $\{\balpha^{(j)}\}_{j=1}^\ell$, then $\bc$ also belongs to the cone $\mathcal{C}_\sage\left(\tilde{\balpha}^{(1)}, \allowbreak \dots, \allowbreak \tilde{\balpha}^{(\ell)} \right)$ for all sufficiently small perturbations $\tilde{\balpha}^{(j)}$ of each $\balpha^{(j)}$.  As the SAGE cone $\mathcal{C}_\sage$ is a direct sum of AM/GM exponential cones $\mathcal{C}_\age$, robustness of each of the component AM/GM exponential cones $\mathcal{C}_\age$ implies the robustness of $\mathcal{C}_\sage$.  The following proposition shows that cones of AM/GM exponentials are robust to small perturbations of the exponents for a broad class of exponents.  The basic observation underlying this result is that the exponents $\{\balpha^{(j)}\}_{j=1}^\ell$ appear only as linear constraints in the characterization of the cones $\mathcal{C}_\sage$ and $\mathcal{C}_\age$ based on \eqref{eq:cage} and Proposition~\ref{prop:sagecone}.

\begin{proposition} \label{prop:robustness}
Fix a set of exponents $\{\balpha^{(j)}\}_{j=0}^{\ell} \subset \R^n$, and let $A$ denote the $\R^{(\ell+1) \times n}$ matrix with the $\balpha^{(j)}$'s as the columns.  Let $(\bc, \beta) \in \R^{\ell}_+ \times \R$ be any point in the interior of the cone $\mathcal{C}_\age(\balpha^{(1)},\dots,\balpha^{(\ell)}; \balpha^{(0)})$.  Suppose $\balpha^{(0)}$ is not contained in the convex hull of any subset of $k \leq n$ of the exponents $\{\balpha^{(j)}\}_{j=1}^{\ell}$.  Then there exists an open set $\mathcal{S} \subset \R^{(\ell+1) \times n}$ containing $A$ with the following property: For any $\tilde{A} \in \mathcal{S}$ with columns $\{\tilde{\balpha}^{(j)}\}_{j=0}^{\ell}$, we have that $(\bc, \beta)$ belongs to $\mathcal{C}_\age(\tilde{\balpha}^{(1)},\dots,\tilde{\balpha}^{(\ell)}; \tilde{\balpha}^{(0)})$.
\end{proposition}

\paragraph{Proof sketch} We provide a brief outline of the proof.  As $(\bc,\beta)$ belongs to the interior of the cone $\mathcal{C}_\age(\balpha^{(1)},\dots,\balpha^{(\ell)}; \balpha^{(0)})$, the entries of $\bc$ must be strictly positive.  Further, we have from Lemma~\ref{lemma:agerelent} that there exists a $\bnu \in \R^\ell_+$ such that $D(\bnu, e\bc) < \beta$ and that $\sum_{j=1}^\ell \balpha^{(j)} \bnu_j = (\ones' \bnu) \balpha^{(0)}$. Due to the assumption that $\balpha^{(0)}$ is not contained in the convex hull of any subset of $k \leq n$ of the exponents $\{\balpha^{(j)}\}_{j=1}^{\ell}$, we have that $\bnu$ must contain at least $n+1$ nonzeros.  Therefore, for all sufficiently small perturbations $\{\tilde{\balpha}^{(j)}\}_{j=0}^{\ell}$, there exists a corresponding $\tilde{\bnu}$ close to $\bnu$ such that $\sum_{j=1}^\ell (\tilde{\balpha}^{(j)} - \tilde{\balpha}^{(0)}) \tilde{\bnu}_j = 0$.  (If some of the coordinates of the original $\bnu$ are zero, then there exists a nearby $\tilde{\bnu}$ with zeros in the same coordinates satisfying $\sum_{j=1}^\ell (\tilde{\balpha}^{(j)} - \tilde{\balpha}^{(0)}) \tilde{\bnu}_j = 0$ for all sufficiently small perturbations $\{\tilde{\balpha}^{(j)}\}_{j=0}^{\ell}$.)  The relative entropy function $D(\cdot,e\bc)$ is continuous with respect to the first argument if the second argument $\bc$ is fixed and contains strictly positive entries.  Consequently, we have that $D(\tilde{\bnu},e \bc) < \beta$. $\endproof$

This result demonstrates that the bounds $f_\sage^{(0)}$ \eqref{eq:sagelowerboundp} and $f_\sage^{(0,1)}$ \eqref{eq:sagelowerboundpq} based on SAGE decompositions are robust to small perturbations of the exponents in the underlying SPs; one can also carry out a similar analysis to show that the bounds $f_\sage^{(p)}$ \eqref{eq:sagelowerboundp} and $f_\sage^{(p,q)}$ \eqref{eq:sagelowerboundpq} corresponding to higher levels of the hierarchies are robust to small perturbations of the exponents in the underlying SPs.

\begin{example} \label{example:robust}
Consider the following signomial obtained by perturbing the exponents of the signomial \eqref{eq:sig1} in Example~\ref{example:unc} (but leaving the coefficients unchanged):
\small
\begin{eqnarray*}
f(\bx) &=& 10 \exp\{10.2070 \bx_1 + 0.0082 \bx_2 -0.0039 \bx_3\} \\ && + 10 \exp\{-0.0081 \bx_1 + 9.8024 \bx_2 -0.0097 \bx_3\} \\ && + 10 \exp\{0.0070 \bx_1 - 0.0156 \bx_2 + 8.1923 \bx_3\} \\ && - 14.6794 \exp\{1.5296\bx_1 +  1.0927 \bx_2 + 1.3441 \bx_3\} \\ && - 7.8601 \exp\{1.0750\bx_1 +  1.9108 \bx_2 + 1.6339 \bx_3\} \\ && + 8.7838 \exp\{1.0513 \bx_1 + 0.0571 \bx_2 + 1.6188 \bx_3\}.
\end{eqnarray*}
\normalsize
The SAGE lower bound is tight for the perturbed signomial specified here, with the optimal value being equal to $-0.9458$, and the optimal solution being $\bx^\star = (-0.3016, \allowbreak -0.2605, \allowbreak -0.4013)'$.  Recall that the SAGE lower bound was also tight for the signomial \eqref{eq:sig1} of Example~\ref{example:unc}, with the optimal value being equal to $-0.9747$ and the optimal solution being $\bx^\star = (-0.3020, \allowbreak -0.2586, -0.4010)'$.  Hence, this example provides numerical evidence for the robustness of our relative entropy relaxation methods with respect to small perturbations of the exponents.
\end{example}

Approaching this SP as a polynomial optimization problem clearly illustrates the shortcomings of that viewpoint, as small changes in the exponents can lead to very different polynomials after clearing denominators in the exponents.  In turn, the quality of the bounds and amount of computation required to obtain them via SDP relaxations based on sum-of-squares techniques can vary dramatically for small changes in the exponents.  However, relative entropy relaxations for SPs based on SAGE decompositions are well-behaved under such small perturbations of the exponents.

\section{Discussion} \label{sec:discussion}

A number of directions for further investigation arise from this work, and we mention some of these here.

\paragraph{Implications for polynomial optimization} 
The methods developed in this paper may offer a useful addition to SDP relaxation frameworks based on sum-of-squares decompositions for polynomial optimization problems over the nonnegative orthant (equivalently, polynomial optimization problems involving even polynomials).  For example, the Motzkin polynomial $p(\by_1, \by_2, \by_3) = \by_1^2 \by_2^4 + \by_1^4 \by_2^2 + \by_3^6 - 3 \by_1^2 \by_2^2 \by_3^2$ is certified as being globally nonnegative via the AM/GM inequality -- equivalently, $\tilde{p}(\bx_1,\bx_2,\bx_3) = p(\exp\{\bx_1\}, \exp\{\bx_2\}, \exp\{\bx_3\})$ is an AM/GM-exponential -- although this polynomial cannot directly be expressed as a sum-of-squares.  More generally, the methods in this paper may offer effective certificates of global nonnegativity for certain ``sparse'' polynomials of large degree, which are sometimes referred to as ``fewnomials'' in the literature \cite{Kho1991}; this point is also exploited in \cite{GhaM2012}, where the authors employ GPs to compute lower bounds on certain multivariate polynomials more efficiently than SDPs based on sum-of-squares techniques. In fact one can combine the SDP and the relative entropy relaxation frameworks, based on sum-of-squares and SAGE decompositions respectively, to obtain alternative hierarchies for polynomial optimization problems.  Such a combined hierarchy would provide bounds that are at least as good as those produced by SDP relaxations, although this improvement comes at an increased computational cost; it is of interest to investigate and quantify these comparisons.

\paragraph{Efficient transcendental representations of semialgebraic sets} Consider the set $S_d = \{(a,b) \in \R^2 ~|~ x^{2d} + ax^2 + b \geq 0 ~ \forall x \in \R\}$ for $d \in \mathbb{Z}_{++}$.  This set is convex and semialgebraic for each $d \in \mathbb{Z}_{++}$, and it is SDP representable (i.e., expressible as the projection of a slice of the cone of $w_d \times w_d$ positive semidefinite matrices, where $w_d \in \mathbb{Z}_{++}$ may depend on $d$) based on the observation that a nonnegative univariate polynomial can be expressed as a sum-of-squares.  However, these sum-of-squares decompositions consist of many terms for large $d$; thus, we do not know of SDP descriptions of $S_d$ in which $w_d$ does not grow with $d$.  On the other hand, $S_d$ can be represented efficiently via relative entropy inequalities as $S_d = \{(a,b) \in \R \times \R_+ ~|~ \exists \bnu \in \R^2_+ ~\mathrm{s.t.}~ D(\bnu, e (1,b)')\leq a, ~ (d-1)\bnu_1 = \bnu_2 \}$.  Note that the size of this description does not grow with $d$. Consequently, there may be semialgebraic sets (indeed, SDP representable sets) that are more efficiently specified via transcendental representations.  Underlying this discussion is the insight that the AM/GM inequality offers a different proof system than sum-of-squares methods for certifying nonnegativity.  Specifically, the length of a sum-of-squares proof of the AM/GM inequality for a \emph{fixed} set of rational weights can be very large depending on the denominators of the rational weights.  On the other hand, relative entropy methods offer an effective approach for \emph{computing} a set of weights that certifies nonnegativity via the AM/GM inequality.  Investigating the respective power of these different proof systems may yield new insights into the possibility of efficient transcendental representations of convex semialgebraic sets.

\paragraph{Effective bounds on quality of approximation} A third direction is to provide useful bounds on the size of a relaxation required in order to achieve a specified quality of approximation to the global minimum.  An important step towards addressing this question is to achieve a deeper understanding of the effect of perturbations in the exponents of an SP on the eventual bounds computed at a certain level of our hierarchy.  As discussed in Section~\ref{subsec:completeness}, the techniques employed in our proofs of Theorems~\ref{theorem:unc} and \ref{theorem:con} do not provide guidance on the effects of such perturbations, and consequently a different set of ideas may be required to address this question.

\section*{Acknowledgements}
We wish to thank Bill Helton, Pablo Parrilo, and Bernd Sturmfels for helpful conversations at various stages of this project.  This work was supported in part by the following grants: National Science Foundation Career award CCF-1350590 and Air Force Office of Scientific Research grant FA9550-14-1-0098.

\appendix
\section{Appendix} We begin by briefly stating the relevant representation theorems from the real algebraic geometry literature -- the one we use is due to Krivine \cite{Kri1964}; see \cite{Put1993,Sch1991,Ste1974} for other prominent examples.  Then we describe a transformation of SPs with rational exponent vectors to polynomial optimization problems.  Finally, we put these steps together to give proofs of Theorems~\ref{theorem:unc} and \ref{theorem:con}.

\subsection{Representation Theorems for Archimedean Modules over Preprimes} \label{app:rep} We primarily draw from the wonderful exposition in \cite{Mar2008}.  Let $\R[\bY_1,\allowbreak\dots,\allowbreak\bY_\ell]$ denote the polynomial ring in indeterminates $\bY$ with real coefficients, and let $\Q$ denote the rational numbers.  The first object that plays a critical role in our development is a preprime in a polynomial ring:
\begin{definition} \label{defn:preprime}
A \emph{preprime} $P \subset \R[\bY_1,\dots,\bY_\ell]$ is a subset containing $\Q$ that is closed under addition and multiplication.  Further, a preprime $P$ is \emph{Archimedean} if for each $r \in \R[\bY_1,\dots,\bY_\ell]$ there exists $m \in \mathbb{Z}$ such that $r + m \in P$.
\end{definition}

For a preprime $P \subset \R[\bY_1,\dots,\bY_\ell]$, the \emph{ring of bounded elements of $P$ with respect to $\R[\bY_1,\dots,\allowbreak\bY_\ell]$} is defined as:
\begin{equation*}
H_P = \{r \in \R[\bY_1,\dots,\bY_\ell] ~|~ \exists m \in \mathbb{Z} ~\mathrm{s.t.}~ m \pm r \in P \}.
\end{equation*}
From \cite[Prop 5.1.3]{Mar2008}, the preprime $P$ is Archimedean if and only if $H_P = \R[\bY_1,\dots,\allowbreak\bY_\ell]$.  The next algebraic structure that is central to our discussion is a module over a preprime:
\begin{definition} \label{defn:module}
Let $P \subset \R[\bY_1,\dots,\bY_\ell]$ be a preprime.  Then $M \subset \R[\bY_1,\allowbreak\dots,\allowbreak\bY_\ell]$ is a \emph{$P$-module} if it is closed under addition, is closed under multiplication by an element of $P$, and contains $1$.  As with preprimes, a $P$-module $M$ is \emph{Archimedean} if for each $r \in \R[\bY_1,\dots,\bY_\ell]$ there exists $m \in \mathbb{Z}$ such that $r + m \in M$.
\end{definition}

Note that $1 \in M$ for a $P$-module $M$ implies that $P \subset M$.  Consequently, if a preprime $P$ is Archimedean, then any $P$-module $M$ is also Archimedean.  Finally, $P$ itself is a $P$-module.  The following result due, in various forms, to several authors is the main foundation for our proofs of Theorems~\ref{theorem:unc} and \ref{theorem:con}:
\begin{theorem} \cite{Kri1964,Mar2008} \label{theorem:psatz}
Let $P \subset \R[\bY_1,\dots,\bY_\ell]$ be an Archimedean preprime and let $M$ be a $P$-module.  Let $\mathcal{K}_M$ denote the set of points in $\R^\ell$ on which every element of $M$ is nonnegative:
\begin{equation*}
\mathcal{K}_M = \{\by \in \R^\ell ~|~ r(\by) \geq 0 ~ \mathrm{for~all}~ r \in M \}.
\end{equation*}
If $s \in \R[\bY_1,\dots,\bY_\ell]$ is strictly positive on the set of points $\mathcal{K}_M$, then $s \in M$:
\begin{equation*}
s(\by) > 0 ~ \mathrm{for~all}~ \by \in \mathcal{K}_M \Rightarrow s \in M.
\end{equation*}
\end{theorem}
The converse direction, i.e., $s \in M$ implies $s(\by) \geq 0 ~ \mathrm{for~all}~ \by \in \mathcal{K}_M$, follows directly from the definition of $\mathcal{K}_M$.  Results of the form of Theorem~\ref{theorem:psatz} are known as representation theorems because the $P$-module $M$ is usually constructed explicitly, and the positivity of a polynomial $s$ on the nonnegativity set $\mathcal{K}_M$ associated to $M$ implies that $s$ has an explicit representation.

\subsection{Implicitization and Parametrization of SPs with Rational Exponents}
In this section, we present an elementary result that allows us to transform certain families of SPs with rational exponents to polynomial optimization problems over the nonnegative orthant.  This result is employed in our proofs of Theorems~\ref{theorem:unc} and \ref{theorem:con}.  Specifically, consider the problem of verifying positivity of $\inf_{\bx \in \mathcal{S}} f(\bx)$ for $\mathcal{S} = \{\bx \in \R^n ~|~ g_i(\bx) \geq 0, ~ i=1,\dots,m\}$, where $f(\bx),\{g_i(\bx)\}_{i=1}^m$ are signomials with respect to exponents $\{\balpha^{(j)}\}_{j=1}^\ell \subset \Q^n$.  Letting $f(\bx) = \sum_{j=1}^\ell {\bc_j^{(f)}}' \exp\{{\balpha^{(j)}}' \bx\}$ and letting each $g_i(\bx) = \sum_{j=1}^\ell {\bc_j^{(g_i)}}' \exp\{{\balpha^{(j)}}' \bx\}$, the positivity of $\inf_{\bx \in \mathcal{S}} f(\bx)$ with $\mathcal{S} = \{\bx \in \R^n ~|~ g_i(\bx) \geq 0, ~ i=1,\dots,m\}$ is equivalent to $\inf_{\by \in \mathcal{T}_{++}} {\bc^{(f)}}' \by > 0$ for
\begin{equation}
\begin{aligned}
\mathcal{T}_{++} = &\{\by \in \R^\ell ~|~ {\bc^{(g_i)}}' \by \geq 0, ~ i=1,\dots,m \} \\ &\cap \{\by \in \R^\ell_{++} ~|~ \exists \bx \in \R^n ~\mathrm{s.t.}~ \by_j = \exp\{{\balpha^{(j)}}' \bx\}, ~ j=1,\dots,\ell\}.
\end{aligned} \label{eq:beforepolyt}
\end{equation}

\subsubsection{Implicitization} For rational exponents $\{\balpha^{(j)}\}_{j=1}^\ell \in \Q^n$, the set $\{\by \in \R^\ell_{++} ~|~ \exists \bx \in \R^n ~\mathrm{s.t.}~ \by_j = \exp\{{\balpha^{(j)}}' \bx\}, ~ j=1,\dots,\ell\}$ can be implicitized via polynomial equations as follows.  Suppose without loss of generality that the first $n$ of the exponent vectors $\{\balpha^{(j)}\}_{j=1}^\ell$ are linearly independent.  Consequently, the remaining exponents $\{\balpha^{(j)}\}_{j=n+1}^\ell$ can be expressed as linear combinations with \emph{rational} coefficients of the first $n$ exponents.  As a result, each $\by_j$ for $j = n+1,\dots,\ell$ can be expressed as $\by_j = \prod_{k=1}^n \by_k^{\bgamma^{(j)}_k}$ for rational vectors $\bgamma^{(j)} \in \Q^n, ~ j=n+1,\dots,\ell$.  By moving terms involving negative powers from the right to the left side of each of these equations and then suitably clearing denominators in the exponents (valid manipulations over the positive orthant $\R^\ell_{++}$), the set $\mathcal{T}_{++}$ from \eqref{eq:beforepolyt} can be characterized as
\begin{equation}
\begin{aligned}
\mathcal{T}_{++} = &\{\by \in \R^\ell ~|~ {\bc^{(g_i)}}' \by \geq 0, ~ i=1,\dots,m \} \\ &\cap \{\by \in \R^\ell_{++} ~|~ v_j(\by) - w_j(\by) = 0, ~ j=n+1,\dots,\ell \}
\end{aligned} \label{eq:polyt}
\end{equation}
for monomials $v_j(\by), w_j(\by), j=n+1,\dots,\ell$.  The restriction to the positive orthant $\R^\ell_{++}$ in the second set here is crucial as the manipulations we describe are only valid over the positive orthant.

In our appeal to Theorem~\ref{theorem:psatz}, we consider the positivity of $\inf_{\by \in \mathcal{T}_+} {\bc^{(f)}}' \by$ for
\begin{equation}
\begin{aligned}
\mathcal{T}_+ = &\{\by \in \R^\ell ~|~ {\bc^{(g_i)}}' \by \geq 0, ~ i=1,\dots,m \} \\ &\cap \{\by \in \R^\ell_{+} ~|~ v_j(\by) - w_j(\by) = 0, ~ j=n+1,\dots,\ell \}.
\end{aligned} \label{eq:polytclosure}
\end{equation}
However, the implicitization operation is, in general, only valid over the positive orthant $\R^\ell_{++}$.  Consequently, we require that the closure of the set $\mathcal{T}_{++}$ is equal to $\mathcal{T}_+$, as this would imply that $\inf_{\by \in \mathcal{T}_+} {\bc^{(f)}}' \by = \inf_{\by \in \mathcal{T}_{++}} {\bc^{(f)}}' \by$.  The next two results give conditions under which $\mathcal{T}_+ = \mathrm{cl}(\mathcal{T}_{++})$; these correspond to the assumptions in Theorems~\ref{theorem:unc} and \ref{theorem:con}.

\begin{lemma} \label{lemma:conimplicit}
Fix a set of exponents $\{\balpha^{(j)}\}_{j=1}^\ell \in \Q^n$ with $\{\balpha^{(j)}\}_{j=1}^n$ being linearly independent, and consider an SP with transformed constraint sets $\mathcal{T}_{++}$ and $\mathcal{T}_+$ as defined in \eqref{eq:polyt}, \eqref{eq:polytclosure} for some set of coefficients $\bc^{(g_i)} \in \R^\ell, i=1,\dots,m$.  If $\{\by \in \R^\ell ~|~ {\bc^{(g_i)}}' \by \geq 0, ~ i=1,\dots,m \} \subset \R^\ell_{++}$, then $\mathrm{cl}(\mathcal{T}_{++}) = \mathcal{T}_+$.
\end{lemma}

\begin{proof}
One can check that $\mathcal{T}_{++} = \mathcal{T}_+$.  The result then follows from the observation that $\mathcal{T}_+$ is a closed set.
\end{proof}

\begin{lemma} \label{lemma:uncimplicit}
Fix a set of exponents $\{\balpha^{(j)}\}_{j=1}^\ell \in \Q^n$ with $\{\balpha^{(j)}\}_{j=1}^n$ being linearly independent.  If the remaining vectors $\{\balpha^{(j)}\}_{j=n+1}^\ell$ lie in the cone generated by $\{\balpha^{(j)}\}_{j=1}^n$, then $\mathrm{cl}(\mathcal{T}_{++}) = \mathcal{T}_+$, where these sets are defined in \eqref{eq:polyt}, \eqref{eq:polytclosure}.
\end{lemma}

\begin{proof}
As $\{\balpha^{(j)}\}_{j=n+1}^\ell$ lie in the cone generated by $\{\balpha^{(j)}\}_{j=1}^n$, the constraint set $\mathcal{T}_{++} = \{\by \in \R^\ell_{++} ~|~ v_j(\by) - w_j(\by) = 0, ~ j=n+1,\dots,\ell \}$ is specified via monomials $v_j(\by),w_j(\by), j=n+1,\dots,\ell$ with the property that each $v_j(\by)$ is a monomial consisting \emph{only} of the variable $\by_j$ (i.e., just $\by_j$ raised to a positive integer) and each $w_j(\by)$ is a monomial consisting of (a subset of) the variables $\by_1,\dots,\by_n$.  This follows directly from the assumption that $\{\balpha^{(j)}\}_{j=n+1}^\ell$ lie in the cone generated by $\{\balpha^{(j)}\}_{j=1}^n$, and from the construction of the monomials $v_j(\by),w_j(\by)$ preceding \eqref{eq:polyt}.  Consequently, one can check that for each point in $\mathcal{T}_+$ that lies on a face of the nonnegative orthant $\R^\ell_{+}$, the point can be approached as a limit of a sequence in $\mathcal{T}_{++}$.  Hence, $\mathrm{cl}(\mathcal{T}_{++}) = \mathcal{T}_+$.
\end{proof}

After the implicitization step and the appeal to Theorem~\ref{theorem:psatz} to obtain representations of positive polynomials, we finally derive a conclusion in terms of signomials as follows.  We apply the reverse transformation $\by_j \leftarrow \exp\{{\balpha^{(j)}}' \bx\}$ for each $j=1,\dots,\ell$, which is a valid parametrization of the system of polynomial equations $v_j(\by) - w_j(\by) = 0, j=n+1,\dots,\ell$ for $\by \in \R^\ell_{++}$.  The validity of this step relies on the fact that $\{\balpha^{(j)}\}_{j=1}^n$ are linearly independent.

\subsubsection{SAGE functions and redundant constraints as polynomials} \label{subsubsec:sageaspoly} The above discussion implies that a nonnegative signomial $\sum_{j=1}^\ell \bc_j \exp\{{\balpha^{(j)}}' \bx\}$ defined with respect to rational exponents $\{\balpha^{(j)}\}_{j=1}^\ell \subset \Q^n$ can also be viewed as a linear function $\bc' \by$ that is nonnegative over $\mathrm{cl}(\mathcal{T}_{++})$, where $\mathcal{T}_{++} = \{\by \in \R^\ell_{++} ~|~ v_j(\by) = w_j(\by), j=n+1,\dots,\ell \}$ is constructed as in the description preceding \eqref{eq:polyt}.  A SAGE function is simply a nonnegative linear function over $\mathrm{cl}(\mathcal{T}_{++})$ with an efficiently computable nonnegativity certificate.  In a similar vein, nonnegative signomials (SAGE functions) defined with respect to exponents in $E_p(\{\balpha^{(j)}\}_{j=1}^\ell)$ can be viewed as degree-$p$ polynomials that are nonnegative (efficiently certified as being nonnegative) over $\mathrm{cl}(\mathcal{T}_{++})$ and vice-versa.  As a result of this correspondence, we use overloaded terminology to refer to SAGE functions both in the domain of signomials (distinguished by the use of the variable $\bx$) and in the domain of polynomials (using the variable $\by$).  In an analogous manner, we also view elements of the set of redundant constraints $R_q(C)$ as degree-$q$ polynomials in the variables $\by_1,\dots,\by_\ell$ that are constrained to be nonnegative over $\mathrm{cl}(\mathcal{T}_{++})$.

\subsection{Proof of Theorem~\ref{theorem:con}} \label{app:conproof} Let $f(\bx) = \sum_{j=1}^\ell {\bc_j^{(f)}} \exp\{{\balpha^{(j)}}' \bx\}$ and let each $g_i(\bx) = \sum_{j=1}^\ell {\bc_j^{(g_i)}} \exp\{{\balpha^{(j)}}' \bx\}$ for $i=1,\dots,m$.  Set $\by_j = \exp\{{\balpha^{(j)}}' \bx\}$ for $j=1,\dots,\ell$.  The assumption that constraints of the form $U \geq \exp\{{\balpha^{(j)}}' \bx\} \geq L$ for each $j=1,\dots,\ell$ with $U,L \in \R_{++}$ are explicitly specified in $C$ implies that $\{\by \in \R^\ell ~|~ {\bc^{(g_i)}}' \by \geq 0, ~ i=1,\dots,m \} \subset \R^\ell_{++}$.  Hence, we appeal to Lemma~\ref{lemma:conimplicit} to conclude that $\inf_{\by \in \mathcal{T}_{++}} {\bc^{(f)}}' \by = \inf_{\bx \in \K_C} f(\bx) > 0$ is equivalent to $\inf_{\by \in \mathcal{T}_+} {\bc^{(f)}}' \by > 0$, where the set $\mathcal{T}_+ = \{\by \in \R^\ell ~|~ {\bc^{(g_i)}}' \by \geq 0, ~ i=1,\dots,m \} \cap \{\by \in \R^\ell_+ ~|~ v_j(\by) - w_j(\by) = 0, ~ j=n+1,\dots,\ell \}$ for suitable monomials $v_j,w_j$ \eqref{eq:polytclosure}.

Let $\tilde{P} \subset \R[\bY_1,\allowbreak\dots,\allowbreak\bY_\ell]$ be the preprime generated by $\R_+$, $\by_1,\dots,\by_\ell$, and ${\bc^{(g_1)}}' \by, \allowbreak \dots, \allowbreak {\bc^{(g_m)}}' \by$.  We note that $\tilde{P}$ is Archimedean as it contains $\by_j$ and $U - \by_j$ for each $j=1,\dots,\ell$, and hence the generators of the preprime are in the ring of bounded elements.  Consider the preprime $P = \tilde{P} + I$, where $I$ is the ideal generated by the polynomials $v_j(\by) - w_j(\by)$ for $j=n+1,\dots,\ell$.  As $\tilde{P}$ is Archimedean, so is $P$.

Next consider the set $\tilde{M}$ formed by taking finite sums of elements from the set $\{s(\by) h(\by) ~|~ s(\by) \in \sage(E_p(\{\balpha^{(j)}\}_{j=1}^\ell)), ~ h(\by) \in R_q(C), ~ p,q \in \mathbb{Z}_+\}$; see Section~\ref{subsubsec:sageaspoly} for discussion on this overloaded notation.  The set of SAGE functions is closed under addition, and also under multiplication by posynomials; thus, one can check that $\tilde{M}$ is a $\tilde{P}$-module.  Consequently, $M = \tilde{M} + I$ is a $P$-module.  Further, $M$ is Archimedean as $P$ is Archimedean.  Finally, we observe that the nonnegativity set $\K_{M} = \{\by \in \R^\ell ~|~ r(\by) \geq 0 ~ \mathrm{for~all}~ r \in M\}$ is equal to $\mathcal{T}_+$ as the constraints specified by the polynomials in $M$ are implied by those that specify the constraints in $\mathcal{T}_+$.  Hence, we apply Theorem~\ref{theorem:psatz} to conclude that ${\bc^{(f)}}' \by \in M$.  By substituting back $\by_j \leftarrow \exp\{{\balpha^{(j)}}' \bx\}$, the term corresponding to the ideal $I$ vanishes and we obtain the desired result. $\endproof$

\subsection{Proof of Theorem~\ref{theorem:unc}}  \label{app:uncproof} We begin with the following result on a constrained analog of Polya's theorem \cite{Mar2008,Pol1928} regarding representations of forms that are positive on the simplex; a generalization of this statement is proved in \cite{DicP2014}, but the specialization stated here suffices for our purposes.

\begin{proposition}\cite{DicP2014} \label{prop:unchom}
Let $u_k(\by)$ for $k=0,\dots,s$ be a collection of homogenous polynomials in $\R[\bY_1,\dots,\bY_\ell]$.  Let $\mathcal{V} = \{\by \in \R^\ell_+ ~|~ u_k(\by) = 0, k=1,\dots,s\}$.  If $u_0(\by) > 0$ for all $\by \in \mathcal{V} \backslash \{\bzero\}$, then there exists $p \in \mathbb{Z}_+$ such that $(\sum_{j=1}^\ell \by_i)^p u_0(\by) = \tau(\by) + \theta(\by)$ where $\tau(\by)$ is a polynomial consisting of all nonnegative coefficients and $\theta(\by)$ lies in the ideal generated by the polynomials $\{u_k(\by)\}_{k=1}^s$.
\end{proposition}

We use this result after suitably transforming our unconstrained SP to a polynomial optimization problem over the nonnegative orthant.

\paragraph{Proof of Theorem~\ref{theorem:unc}} The main idea in this proof is to transform the signomial $f(\bx) = \sum_{j=1} \bc_j \exp\{{\balpha^{(j)}}' \bx\}$ appropriately so that we can apply Proposition~\ref{prop:unchom}.

The unconstrained SP $\inf_{\bx \in \R^n} f(\bx)$ can be transformed to $\inf_{\by \in \mathcal{T}_+} \bc' \by$ based on Lemma~\ref{lemma:uncimplicit}, where $\mathcal{T}_+ = \{\by \in \R^\ell_+ ~|~ v_j(\by) = w(\by), j=n+1,\dots,\ell\}$ and the construction of the monomials $v_j(\by),w_j(\by)$ is described in the discussion preceding \eqref{eq:polyt}.  We note that the monomials $v_j(\by),w_j(\by)$ for $j=n+1,\dots,\ell$ satisfy several properties:
\begin{enumerate}
\item[(i)] We have that $v_{n+1}(\by) = \by_{n+1}$ and $w_{n+1}(\by) = 1$ because $\balpha^{(n+1)} = \bzero$.

\item[(ii)] For each $j=n+2,\dots,\ell$, the monomial $v_j(\by)$ is \emph{only} a function of $\by_{j}$ and the monomial $w_j(\by)$ is \emph{only} a function of $\by_1,\dots,\by_n$.  This follows because each $\balpha^{(j)}$ is a convex combination with rational weights of the exponents $\{\balpha^{(i)}\}_{i=1}^{n+1}$ for $j=n+2,\dots,\ell$.

\item[(iii)] Finally, $\mathrm{deg}(v_j) > \mathrm{deg}(w_j)$ for each $j=n+2,\dots,\ell$, where $\mathrm{deg}(\cdot)$ denotes the degree of a polynomial.  This is a consequence of the assumption that $\balpha^{(j)} \notin \mathrm{conv}(\{\balpha^{(i)}\}_{i=1}^{n})$ for each $j=n+2,\dots,\ell$.
\end{enumerate}

Now define modified monomials:
\begin{equation}
\tilde{w}_j(\by) = w_j(\by) \by_{n+1}^{\mathrm{deg}(v_j)-\mathrm{deg}(w_j)}, ~ j=n+2,\dots,\ell, \label{eq:modifiedmon}
\end{equation}
and consider the set $\tilde{\mathcal{T}}_+ = \{\by \in \R^\ell_+ ~|~ v_j(\by) = \tilde{w}_j(\by), j=n+2,\dots,\ell\}$.  The difference between $\tilde{\mathcal{T}}_+$ and $\mathcal{T}_+$ is that each $w_j(\by)$ is replaced by $\tilde{w}_j(\by)$, and there is no constraint corresponding to $j=n+1$ (i.e., the constraint $\by_{n+1} = 1$ from property $(i)$ above is removed).  Further, note that the polynomials $v_j(\by) - \tilde{w}_j(\by)$ for $j=n+2,\dots,\ell$ are homogenous polynomials based on property $(iii)$ above.

As the polynomials $v_j(\by) - \tilde{w}_j(\by)$ for $j=n+2,\dots,\ell$ are homogenous, we intend to apply Proposition~\ref{prop:unchom} to the polynomial optimization problem $\inf_{\by \in \tilde{\mathcal{T}}_+} \bc' \by$.  In order to do so, we need to verify that $\bc' \by > 0$ for all $\by \in \tilde{\mathcal{T}}_+ \backslash \{\bzero\}$.

First, suppose that $\by \in \tilde{\mathcal{T}}_+ \backslash \{\bzero\}$ with $\by_{n+1} > 0$.  As the polynomials $v_j(\by) - \tilde{w}_j(\by)$ for $j=n+2,\dots,\ell$ and the linear function $\bc' \by$ are homogenous, we may scale $\by$ so that $\by_{n+1} = 1$ since this not affect our conclusions about the positivity of $\bc' \by$.  Setting $\by_{n+1}=1$ corresponds to verifying the positivity of $\bc' \by$ over $\mathcal{T}_+$.  As $\inf_{\by \in \mathcal{T}_+} \bc' \by = \inf_{\bx \in \R^n} f(\bx) > 0$, we have that $\bc' \by > 0$ for all $\by \in \tilde{\mathcal{T}}_+ \backslash \{\bzero\}$ whenever $\by_{n+1} \neq 0$.

Next, suppose that $\by \in \tilde{\mathcal{T}}_+ \backslash \{\bzero\}$ with $\by_{n+1} = 0$.  The definition of $\tilde{\mathcal{T}}_+$ directly implies that each $\by_j = 0$ for $j=n+2,\dots,\ell$, because $v_j(\by)$ for $j=n+2,\dots,\ell$ is only a monomial consisting of $\by_j$ raised to a positive integer (from property $(ii)$ above) and $\tilde{w}_j(\by)$ consists of $\by_{n+1}$ raised to a positive integer (from property $(ii)$ above and from \eqref{eq:modifiedmon}).  Consequently, we need to verify the positivity of $\bc' \by$ with $\by \in \R^\ell_+ \backslash \{\bzero\}$ and with $\by_j=0$ for $j=n+1,\dots,\ell$.  Since $\bc_j > 0$ for $j=1,\dots,n$ by assumption, it follows that $\bc' \by > 0$ for all $\by \in \tilde{\mathcal{T}}_+ \backslash \{\bzero\}$ with $\by_{n+1} = 0$.

Combining these observations, we appeal to Proposition~\ref{prop:unchom} to conclude that $(\sum_{j=1}^\ell \by_i)^p \bc' \by = \tau(\by) + \theta(\by)$ where $\tau(\by)$ is a polynomial consisting of all nonnegative coefficients and $\theta(\by)$ lies in the ideal generated by the polynomials $\{v_j(\by) - \tilde{w}_j(\by)\}_{j=n+2}^\ell$.  By substituting $\by_{n+1} \leftarrow 1$ and $\by_j \leftarrow \exp\{{\balpha^{(j)}}' \by\}$ for $j=n+2,\dots,\ell$, and by noting that posynomials are SAGE functions, we have the desired result. $\endproof$


\bibliography{signomial_refs}

\end{document}